\documentclass{amsart}
\usepackage[utf8]{inputenc}
\usepackage{amsmath}
\usepackage{amssymb}
\usepackage{tikz}
\usepackage{graphicx}
\usepackage{float}
\usepackage{framed}
\usepackage[hang,flushmargin]{footmisc}
\usepackage{comment}
\usepackage{ytableau}
\usepackage{hyperref}
\usepackage{cleveref}

\usepackage{amsthm}
\usepackage{mathtools}
\newtheorem{theorem}{Theorem}[section]
\newtheorem{corollary}[theorem]{Corollary}
\newtheorem{lemma}[theorem]{Lemma}
\newtheorem{definition}[theorem]{Definition}
\newtheorem{proposition}[theorem]{Proposition}

\newtheorem{remark}[theorem]{Remark}

\usepackage{scalerel,stackengine}
\stackMath
\newcommand\reallywidehat[1]{%
\savestack{\tmpbox}{\stretchto{%
  \scaleto{%
    \scalerel*[\widthof{\ensuremath{#1}}]{\kern-.6pt\bigwedge\kern-.6pt}%
    {\rule[-\textheight/2]{1ex}{\textheight}}
  }{\textheight}%
}{0.5ex}}%
\stackon[1pt]{#1}{\tmpbox}%
}
\newcommand{\nm}[1]{ || #1 || }
\newcommand{\C}{\mathbb{C}}
\newcommand{\R}{\mathbb{R}}

\newcommand{\Hc}{\mathcal{H}}
\newcommand{\li}[1]{\overline{ #1}}
\newcommand{\Tr}{\mathrm{Tr}}
\newcommand{\p}{\mathfrak{p}}
\newcommand{\g}{\mathfrak{g}}
\newcommand{\kf}{\mathfrak{k}}
\newcommand{\h}{\mathfrak{h}}

\newcommand{\D}{\mathbb{D}}
\newcommand{\Ad}{\mathrm{Ad}}
\newcommand{\ad}{\mathrm{ad}}

\begin{document}

\title{Wehrl inequalities for matrix coefficients of
holomorphic discrete series}
\author{Robin van Haastrecht and Genkai Zhang}
\begin{abstract}
    We prove Wehrl-type $L^2(G)-L^{p}(G)$
    inequalities 
    for matrix
    coefficients
    of vector-valued
    holomorphic discrete series of $G$,  
    for even integers $p=2n$.
        The optimal constant is expressed
    in terms of Harish-Chandra
    formal degrees for the discrete series.
    We  prove the maximizers
    are precisely the reproducing kernels.    
\end{abstract}

\keywords{
Lie groups, 
unitary representations, Hermitian symmetric
spaces, holomorphic discrete 
series, Harish-Chandra formal degree,
reproducing kernels, matrix coefficients, $L^p$-spaces, Wehrl inequality
Toeplitz operators.}

\subjclass[2020]{22E30, 22E45, 32A36, 43A15}


\thanks{Research
by Genkai Zhang partially supported
by the Swedish Research Council (Vetenskapsr\aa{}det). }

\maketitle
\baselineskip 1.15pc

\section{Introduction}
In the present paper we shall
study the $L^2-L^p$ optimal
inequalities for matrix
coefficients for holomorphic discrete series representations of
Hermitian Lie groups.
We start with a brief
introduction 
on the main problem.
\subsection{Background
and Main Problem}

Let $
(G, \pi, \Hc)$
be a unitary irreducible
representation
of a Lie group $G$
and assume that $\pi$ is
a discrete series
relative to a homogeneous
space $G/H$ for a closed
subgroup $H\subset G$,
namely the square
norms of the matrix coefficients
$\langle \pi(g)u, v\rangle, g\in G, u, v\in \Hc$ are well-defined as elements
in $L^2(G/H)$ for a certain
$G$-invariant measure on $G/H$. The matrix coefficients are in $L^\infty$ by the unitarity. It is a natural and
important question to find
the optimal estimates
for the $L^p$-norm
for $p\ge 2$ as it is related to  other questions and concepts.

The most studied
case is when $G$
is the  Heisenberg group 
$G=\mathbb R\rtimes\mathbb C^n$,
and the unitary representation 
$
(G, \pi, \Hc)$
is on 
the Fock space $\Hc =\mathcal F(\mathbb C^n)$,
or on $\Hc=L^2(\mathbb R^n)$
in the Schr\"o{}dinger model.
The relevant optimal estimates
are sometimes called
Wehrl inequalities \cite{weh}.
The matrix coefficients $\langle \pi(g) f, f_0\rangle$,
when restricted to 
$\mathbb C^n
=G/\mathbb R$,
are in the space 
$L^2(\mathbb C^n)$.
The Fock space $\Hc =\mathcal F(\mathbb C^n)$
has a reproducing kernel
$e^{\langle z, w\rangle}$, which
maximize the $L^\infty$-norm
among elements of fixed $L^2$-norm.
Fix $f_0=1$ as
the reproducing kernel
$e^{\langle z, w\rangle}$
at $w=0$ (or the Gaussian function in
the Schr\"o{}dinger model).
For each 
positive  operators $T\ge 0$ of unit trace,  $\Tr \,T=1$, 
the  matrix coefficients
$F(g)  
=\langle T\pi(g) f_0,  \pi(g) f_0 \rangle
=\Tr (T\pi(g)f_0\otimes 
(\pi(g)f_0)^\ast)
$ defines
a probability measure on $\mathbb C^n=G/\mathbb R$,
$\int_{\mathbb C^n}
F(g) dg=1$
by Weyl's Plancherel formula 
(up to a normalization).
 Wehrl \cite{weh} proposed the quantity $-\int F(g) \ln F(g) dg$
as a classical entropy 
corresponding to the quantum entropy
$-\Tr\,  T\ln T $ defined by $T$. 
Wehrl investigated the question when the entropy is minimal. It is easy to see this must happen for some $T = f \otimes  f^*$ a pure tensor, by concavity of the function $-x \ln x$, so it is enough to consider these pure tensors. Wehrl conjectured
the classical entropy is minimal 
for $f =\pi(g)f_0$, a translation of the
function $f_0=1$ (or the Gaussian function in the Schrödinger model)
by an element $g\in G$.
Lieb \cite{Lieb-cmp}
studied a more general
question on  the optimal $L^2(\mathbb C^n)-L^{p}(\mathbb C^n)$ boundedness, $p \ge 2$
for the matrix coefficients $\langle \pi(g) f, f_0 \rangle, 
g\in \mathbb C^n$, and proved
that the maximizers
are precisely achieved by  $f =\pi(g_0)f_0$
for some $g_0\in G$;
  the Wehrl conjecture
becomes an immediate
consequence
by taking the derivative at $p=2$
of the inequality for $p=2$.

When $G$ is a compact semisimple Lie
group
any irreducible representation $(\pi, \Hc)$
is  finite-dimensional,
and there
is also a preferred choice
of the vector $v_0$, 
namely the highest weight
vector or its translates under the action of $G$, similar to the function $f_0=1$ above for the Heisenberg group. The
Schur orthogonality computes
 $L^2$-norms of 
 the matrix coefficients
$\langle \pi(g)f, f_0\rangle,
g\in G$ using the dimension of $\mathcal H$
and it is natural problem to find  $L^2-L^p$ optimal estimates.
In  \cite{DF}
a statement on the $L^2-L^{4}$
optimal estimate was given
with  a sketch of the proof.
For $G=SU(2)$
the Wehrl $L^2-L^p$
inequality \cite{weh}
was proved by 
Lieb and Solovej \cite{LS-acta} 
more than 30 years later.
They also proved
the inequality \cite{LS-cmp}
for
$G=SU(N)$
and
for the symmetric
tensor power $S^m(\mathbb C^N)$
representations of $G$.
They used methods quite
different from the classical
analytic method
\cite{Lieb-cmp}
by introducing
quantum channel operators
and proving more general
results about the
eigenvalue distribution
of these operators.

The next interesting
and challenging case is for real
simple non-compact Lie groups
$G$ and their
discrete series
representations $(\pi, \Hc)$.
Harish-Chandra has
generalized the Schur
orthogonality relations for compact
groups using the formal degree. It
suggests that there
should
also be optimal
$L^2-L^p$ estimates
for the matrix coefficients, $p\ge 2$.
When $G=SU(1, 1)$ Lieb and Solovej \cite{LS-aff}
proved optimal 
$L^2-L^p$ estimates
for the
Bergman space as holomorphic discrete
series representations
of $G=SU(1, 1)$
for even integers $p=2n$
by using direct computations.
This was generalized to all $p\ge 2$
by Kulikov \cite{ku}
using the isoperimetric inequality for
the hyperbolic area
of sublevel sets of
the holomorphic
functions (as sections
of the cotangent bundle with the dual hyperbolic metric).
In all these cases,
$G=\mathbb R\rtimes \mathbb C^n$,
$SU(2)$ 
and $SU(1, 1)$, 
the inequalities
are proved for any general positive
convex function
instead of the $L^{p}$-norm.
A general systematic
treatment is given by
Frank \cite{Fr}.

\subsection{Our Main Results and Methods}

We consider
now 
a Hermitian Lie group $G$
and its 
holomorphic
discrete series
$(\Hc_{\Lambda}, \pi_\Lambda)$
with highest
weight $\Lambda$.
The discrete series
will be realized
as the Bergman space of
$V_\Lambda$-valued 
holomorphic
functions 
on the bounded symmetric domain $D=G/K$
of $G$
with $(V_\Lambda, \tau_\Lambda, K)$
the unitary representation
of $K$ with $K$-highest
weight $\Lambda$. We will write $\tau = \tau_{\Lambda}$ 
in the rest of the text if no confusion would arise. The holomorphic
functions
can be realized
as sections of the 
holomorphic
vector bundle
over $D$
with the
Harish-Chandra realization
of $D$, and
the metric 
on the bundle
can be expressed
as $\langle \tau_\Lambda (B(z, z)^{-1})v, v\rangle$
using the Bergman operator $B(z,z)$;
see Definition
\ref{def-bergm-op}
below.
The tensor product
$V_\Lambda^{\otimes n}$
has an irreducible
component $V_{n\Lambda}$
of multiplicity
one, now let $P=P_{n\Lambda}:
V_\Lambda^{\otimes n}\to V_{n\Lambda}$
be the orthogonal projection.
Write $P(f^{\otimes n})(z)
=
P(f^{\otimes n}(z))
$, the point-wise orthogonal projection.
Our main result
is the following.

\begin{theorem} 
(Theorem \ref{wherlineq} and Corollary 
\ref{wehrlineqcor})
Let $n\ge 2$ be an integer, 
$(V_\Lambda, 
\tau, K)$ 
be  an irreducible representation 
of $K$ with a unit highest
weight vector $v_\Lambda$
and  $\Hc_{\Lambda}$
 the 
holomorphic discrete
series realized as the Bergman space
of $V_\Lambda$-valued
holomorphic functions.
Then 
\begin{equation}
\label{main-1}
\Vert P(f^{\otimes n}) \Vert_{\Hc_{\Lambda}}^2
 \leq c_G^{n-1} \frac{
(
d_{\Lambda}^{\mathrm
H}
)^n}
{d_{n \Lambda}^{\mathrm{H}
}
} 
\Vert f \Vert_{\Hc_{\Lambda}}
^{2n},
\end{equation}
and
\begin{equation}
\label{main-2}
\Vert 
F_f\Vert_{L^{2n}}^{2n}
\leq c_G^{n-1} \frac{(d_{\Lambda}^{\mathrm{H}} )^{n}}{d_{n \Lambda}^{\mathrm{H}}} 
\Vert 
F_f\Vert_{L^{2}}^{2n}
\end{equation}
for $f\in \Hc_\Lambda$
and 
$F_f(g)
:=\langle\pi(g)f, v_\Lambda\rangle$, $g\in G$.
The equality holds if and only if $f(z) =c K(z,w) \tau(k) v_{\Lambda}$ for some 
$w \in D$, $k \in K$,
$c\in \mathbb C$.
\end{theorem}

The precise notation
is found below.
When $\Hc_\Lambda$
is a scalar holomorphic
discrete series
this result is proved in 
\cite{Z-24}.

We 
explain 
briefly
our methods
and some auxiliary results.
First we consider
the $n$-fold tensor power $\Hc_\Lambda^{\otimes n}$
of the discrete
series.
The orthogonal projection $Pf(z)^{\otimes n}$
of $f(z)^{\otimes n}
\in 
V_\Lambda^{\otimes n}
$
onto the highest component 
(also called the Cartan component)
$V_{n\Lambda}
\subset 
\otimes^n V_\Lambda
$
defines a $G$-intertwining
operator onto
the discrete series
$H_{n\Lambda}$.
This follows
from some
general facts
for holomorphic discrete
series \cite{repka}. Thus
there should be an inequality.
The constant in the inequality
is abstractly obtained by the
Harish-Chandra formal degree.
However the constant
is only determined up to normalization, whereas
our Bergman space
is defined by the usual
normalization. 
We then find the exact
formula for the
Harish-Chandra
formal degree by using
the evaluation
of the Selberg Beta integral 
\cite{fako, AAR};
see  \Cref{constanthcvsour} below. As a 
consequence we find also in Theorem
\ref{constantthm}
the formula
for the reproducing kernel under our normalization.
To prove that the maximizers
are achieved by the reproducing
kernel we prove that
they are eigenvectors
of  Toeplitz operators
  \cite{Z-24}
  and that they define
  the bounded point 
  evaluations.
  We finally use 
the earlier results
in  \cite{DF}
about Wehrl inequalities
for compact groups.
However, we realized the proof in \cite{DF} is incomplete
and we provide a full proof in Appendix \ref{cptwehrl}.

For the unit disk $D=SU(1, 1)/U(1)$
we find 
in Theorem 
\ref{improved-ineq}
an improved Wehrl
inequality
with
a precise
extra term added
on the left hand side of
the Wehrl inequality (\ref{main-1});
the extra term involves 
first and second
derivatives of $f$.
Our result might lead to
finding an improved Wehrl
$L^2-L^p$-inequality for the
Bergman space on 
the unit disc \cite{Fr, ku}
and
for the  Fock space 
\cite{FNT}.

\subsection{Further Questions}
There are quite a few open
questions related
to the Wehrl inequality. 
The Wehrl $L^2-L^p$-inequality
for the Bergman
space on
the unit ball
in $\mathbb C^n$, $n\ge 2$
is still open. 
In \cite{LS-cmp} the 
equality is proved for Bergman spaces
of holomorphic sections
of symmetric tangent
bundles on the projective space
$\mathbb P^n=SU(n+1)/U(n)$
using quantum
channels \cite{LS-acta}. 
These channels
can be defined \cite{Z-24} 
for the general
 holomorphic
discrete series
for $SU(n, 1)$.
In \cite{vH-1, vh-2}
the limit formulae
for the functional
calculus of the
channels are found
generalizing
earlier results
of
\cite{LS-cmp}. 
It would be interesting
to study the eigenvalue
distributions
of the channel operators
for other representations
of $SU(n+1)$
and for the non-compact 
group $SU(1, 1)$.
Kulikov \cite{ku}
proved some subtle
properties
about the hyperbolic
area of holomorphic
functions in the Bergman space
using isoperimetric inequalities.
It might
be important 
to study the volumes
of sublevel
sets for holomophic
functions in Bergman space
in higher dimensions
rather than 
isoperimetric 
problems for general
sets.
For a
discrete series $(\Hc, \pi, G)$
of a semisimple Lie
group it seems a rather
challenging problem to 
find the optimal $L^2-L^{2n}$ estimates.

\subsection{Organization
of the paper}
In Section
\ref{sect-herm}
we recall
some necessary known results on Hermitian symmetric spaces $G/K$,
and in Section 
\ref{sect-HDS}
we introduce
holomorphic discrete series representations of $G$ and their realizations as
Bergman spaces of vector-valued holomorphic functions
on $D$. We
find in Section 
\ref{for-deg} the exact formula
for the Harish-Chandra
formal degrees
under our 
(somewhat  standard) normlization of the metric on $G/K$.
The Wehrl equalities
are proved in Section
\ref{Wehrl-ineq}.
An improved Wehrl inequality
for the unit disc is proved
in  \Cref{improved-W}.
In  \Cref{cptwehrl}
we give a complete
 proof for Wehrl inequalities
for compact semisimple Lie groups and 
in Appendix \ref{BPV} we prove
that the bounded point evaluations
for our Bergman space of vector-valued holomorphic functions are
given by the point in $D=G/K$, they
are all needed to prove the Wehrl inequalities
in  \Cref{Wehrl-ineq}.

\subsection{Notation} For the convenience of the reader we add a list of the most common notation in the paper.

\begin{enumerate}
    \item $G$ is a simple Hermitian Lie group and $G/K$ is a Hermitian symmetric space.
    \item $\g$ is the Lie algebra of $G$.
    \item 
    $\g^{\mathbb C} = \p^+  \oplus \kf^{\C} \oplus \p^-$ is the decomposition of the Lie algebra into eigenspaces of a central element of $\kf^{\C}$.
    \item $D = G/K$ is the bounded Hermitian symmetric domain of rank $r$ realized in $\p^+=\mathbb C^N$.
    \item $\Delta$ are the roots of $\g^{\C}$ with respect to the Cartan subalgebra $\h^{\C}$ of $\kf^{\C}$, which is also a Cartan subalgebra of $\g^{\C}$.
    \item $(V_{\Lambda}, \tau_{\Lambda}, K)$ is a representation of $K$ of highest weight $\Lambda$.
    \item $(\Hc_{\Lambda}, \pi_{\Lambda}, G)$ is the holomorphic discrete series of $G$ associated to the representation $(V_{\Lambda}, \tau_{\Lambda}, K)$.
    \item The Haar measure of $G$ is normalized by $\int_G f(g) dg = \int_D \left( \int_K f(x k) d k \right) d \iota(x)$, where $\int_K dk = 1$ and $d \iota$ is defined in (\ref{bergmanmeasure}).
    \item $d_{\Lambda}^{\mathrm{H}}$ and $d_{\Lambda}$ is the formal degree for a holomorphic discrete series $(\Hc_{\Lambda}, \pi_{\Lambda}, G)$, by different normalizations, see (\ref{formaldegreeort}) and(\ref{ourharidegree}).
    \item $P_{\Lambda}$ is the projection onto an irreducible $K$-representation of highest weight $\Lambda$.
    \item $Q_0$ is the projection onto the Cartan component of highest weight $n \Lambda$ $\Hc_{n \Lambda} \subseteq \Hc_{\Lambda}^{\otimes n}$. For $SU(1,1)$, $Q_k$ is the projection onto the irreducible component $\Hc_{\mu + \nu + 2k} \subseteq \Hc_{\mu} \otimes \Hc_{\nu}$.
\end{enumerate}

{\bf Acknowledgements} We would like
to thank Rupert Frank for some stimulating
discussions.

\section{Hermitian symmetric spaces 
realized as bounded domains $D=G/K$}\label{sect-herm}
We recall briefly some known facts
on Hermitian symmetric spaces
and related Lie algebras. 
We shall use the Jordan triple description; see \cite{loos}
and \cite[Chapter 2.5]{sata}. 

\subsection{Hermitian 
Symmetric spaces $G/K$ the  Lie algebras
$\g$ of $G$. }
Let $G$ be a connected 
simple Lie group
of real rank $r$, $K$ its maximal compact subgroup,
and $G/K$ a Hermitian symmetric space of complex dimension $N$.
Let $\g$ be the Lie algebra of $G$ and $\g = \kf + \p$ the Cartan decomposition with Cartan involution $\theta$. Then $\kf$ has one-dimensional center, so $\kf = [\kf,\kf] \oplus \R Z$, where $Z$ generates the center and is normalized
so that  $J:=\text{ad}(Z)$
defines a complex structure on $\p$. This implies the existence of a Hermitian complex structure on the symmetric space $D = G/K$. Let $\h \subseteq \kf$
be a maximal Cartan subalgebra for $\kf$,
then its complexification
$\h^{\mathbb C} 
\subseteq \kf^{\mathbb C}
\subset \g^{\mathbb C}
$ is also a Cartan subalgebra
for  $\g^{\mathbb C}$ since 
$\kf^{\mathbb C}$ 
and 
$ \g^{\mathbb C}$ are of the same
rank.
The roots $\Delta$ of 
$\h^{\mathbb C} $ in 
$\g^{\C}$ are
$\Delta=\Delta_c
\cup \Delta_n$,  where $\Delta_c$
are the compact roots  $\alpha$
with   $\g_{\alpha} \subseteq \kf^{\C}$,
and $\Delta_n$ the non-compact
roots $\alpha$ with $\g_{\alpha} \subseteq \p^{\C}$. We choose an ordering of roots so that $J=\text{ad}(Z)$ acts on  $\Delta_n^{\pm} = \Delta^{\pm}
\cap \Delta_n$ as $\pm i$. For every $\alpha \in \Delta^+$ we fix an $\mathfrak{sl}_2$-triple such that $h_{\alpha} \in i \h$, $e_{\pm \alpha} \in \mathfrak{g}_{\pm \alpha}$ and
$$[h_{\alpha}, e_{\alpha}] = 2 e_{\alpha}, \ \theta(e_{\alpha}) = -e_{-\alpha}, \ [e_{\alpha}, e_{-\alpha}] = h_{\alpha}.$$
We then have the decomposition 
$\g^{\C}=\kf^{\mathbb C} \oplus \p^+ \oplus \p^-$
with $\p^+$ and $\p^-$
being the sum of the non-compact
positive and negative roots,
respectively and  given
by $$ \p^{\pm} = 
\{ v \mp i J v : v \in \p \}.
$$
Note that $\li{\p^+} = \p^-$, 
$$[\p^+,\p^+] = [\p^-,\p^-] = 0,$$
and
$$[\p^+,\p^-] = \kf^{\C}.$$

Denote
$$
D(u,\li{v}) =[u, \li{v}]\in \mathfrak{k}^{\mathbb C}
$$ identified with its action on $\p^+ = \C^N$,
$$
D(u,\li{v}) 
w 
\coloneqq \ad(D(u, \li{v}))(w) = [D(u,\li{v}), w],
\quad u,v,w \in \p^+ = \C^N.
$$
Then the triple product
$D(u,\li{v}) 
w$ is symmetric in $u$ and $w$. Let $Q(u): \p^{-}=\overline{\mathbb C^N}
\to \p^{+}$ and 
$Q(\li{v}): \p^{+}
\to \p^{-}$ be the  quadratic maps
$$Q(u)\li{v} = \frac{1}{2} D(u, \li{v}) u,
\, Q(\li{v})
u = \frac{1}{2} D(\li{v}, u) \li{v},
\, u\in \p^+, \bar{v}\in \p^-,
$$
See \cite{loos}.

Let $\{ \gamma_i \}_{i=1}^r$
be the strongly orthogonal 
non-compact roots starting with
the highest root $\gamma_1$,
where $r$ is the real rank of $G$.
Dete the corresponding  co-roots
and root vectors  of $\gamma_j$ of by
$$h_j = h_{\gamma_j}, \, e_{\pm j} = e_{\pm \gamma_j}$$
chosen as in \cite{fako}
so that $\overline{e_{\pm j}} = e_{\mp j},$ 
and
$e_i$ is a tripotent \cite{loos}, 
$Q(e_i) \bar{e_i}= e_i$.
The root vectors $\{e_i\}_{i=1}^r$ form a frame, i.e. a maximal orthogonal system of primitive tripotents of unit norm, in the sense of \cite[Section 5.1]{loos}.

Let
$$ p \coloneqq (r-1)a + b + 2,\, n_1 = r + a \frac{r(r-1)}{2}.$$
The dimension $N$ is then
$$N = n_1 + rb.$$
Note the integer $p$ can be computed as
$p = \Tr(D(e_1^+, e_1^-)|_{\p^+}) = \Tr(D(e_j^+,e_j^-)|_{\p^+})$ for any $j$. 
Now we normalize the  $K$-invariant Euclidean
inner product on $\mathbb C^N$ by
\begin{equation}\label{Eucl}
\langle v, w \rangle=\langle v, w \rangle_{\p^+} 
:= \frac{1}{p} \Tr(D(v, \li{w}) |_{\p^+}),
\end{equation}
so that $\nm{e_j} = 1$ 
for any $j$ and the $\{e_j\}_{j=1}^r$ are orthogonal.

\subsection{The Harish-Chandra
factorization of $G$
in $G^{\mathbb C}=P^+K^{\mathbb C} P^-
$
and the Bergman operator}

The symmetric space $D = G/K$ can
be realized as a
circular convex bounded
domain in $\C^N = \mathfrak{p}^+$
as follows, also called the Harish-Chandra
realization. 
Consider the natural
inclusion map followed
by the quotient
map
$$
G \hookrightarrow G^{\C} = P^+ K^{\C} P^- \rightarrow G^{\C} / K^{\C} P^- \cong P^+ \cong \mathfrak{p}^+.$$
Then $K$ is mapped into the reference point
$0 \in \mathfrak{p}^+$ and
it induces an injective holomorphic map
and the
Harish-Chandra realization of 
$$ D:=G / K =G\cdot 0 \subseteq \mathfrak{p}^+.$$
To describe the action of $G$ on $D$ we
need some quantities.

\begin{definition}
\label{def-bergm-op}
The Bergman operator is defined as
$$B(x,\li{y})
= I - D(x, \li{y}) + Q(x) Q(\li{y}): \mathbb C^N
\to \mathbb C^N.
$$
\end{definition}

It follows from \cite[Theorem 8.11]{loos}
that the element $B(z, z)^{-1} \in 
K^{\C}$ for $z\in D=G/K\subset \mathbb C^N$
and
$$B(z,\li{z})^{-1}
\coloneqq K^{\C}
\mathrm{- part  \ of \ } \exp(\li{z})\exp(z)$$
under the decomposition $G^{\C} = P^+ K^{\C} P^-$.

We also have another norm on $\mathbb C^N$, the spectral norm $| - |$, such that
$D$ is a unit ball with the norm,
$$D = \{ z \in \mathbb C^N \ | \ |z| < 1 \},$$ see \cite[Theorem 4.1]{loos}. 
Furthermore, we have the following polar
decomposition
$$ D = \{ \Ad(k)(t_1 e_1 + \dots + t_r e_r) \ | \ k \in K, t_i \in [0,1) \};$$
see \cite[Theorem 3.17]{loos}.

We identify the holomorphic
tangent space
$T_z^{(1,0)}(D)
$ of $D\subset \mathbb C^N$
at $z\in D$ 
with $\mathbb C^N$, 
$T_z^{(1,0)}(D) =\mathfrak p^+$.
Denote  $J_g(z) 
= dg(z)$, 
the Jacobian of the holomorphic map $g: D\to D$ in local coordinates,
$$J_g(z): 
\mathbb C^N = T_z^{(1,0)}(D)
\rightarrow T_{g z}^{(1,0)}(D) = \mathbb C^N.$$
The identification of $\mathbb C^N$ with $T_z^{(1,0)}(D)$ is done by realizing $D \subseteq \mathbb C^N$. Now $B(z,\li{z})$ acts on $\mathbb C^N$ by the adjoint action,
and we have the following important
transformation rule \cite[Lemma 2.11]{loos}
\begin{equation}
\label{transformrule}
J_g(z)^* B(g \cdot z, \li{g \cdot z})^{-1} J_g(z) = B(z,\li{z})^{-1}.
\end{equation}
 As $B(0,0) = I$ it then follows directly that
\begin{equation}
\label{Bzzeq}
 B(g \cdot 0, \li{g \cdot 0}) = J_g(0) J_g(0)^*.
\end{equation}
The Jacobian
$J_g(z)$ 
can be obtained from the more
general  canonical automorphy factor $J(g,z)$ 
 \cite[Lemma 5.3]{sata} 
defined by
$$J(g,z) = K^{\C}-\mathrm{part \ of \ } g \cdot \exp(z);$$
we have $J_g(z) = \Ad(J(g,z))$. 
Since elements in $K^{\mathbb C}$
are realized as linear maps on $\mathfrak p^+$
via the adjoint action we can 
identify $J_g(z)$ with $J(g,z)$, 
but it will be clear from context which one is meant.
In particular we have $J_k(z) =J(k, z)= k$.

\section{Holomorphic discrete
series of $G$ realized as Bergman spaces
of vector-valued holomorphic functions on $D$}
\label{sect-HDS}
\subsection{Bergman space
of holomorphic functions on $D$.
Invariant measure}
Let  $dm(z)$
be the Lebesgue measure
defined by the inner product (\ref{Eucl}).
The Bergman space of
holomorphic functions $f(z)$ on $D$
such that
$$
\int_D |f(z)|^2 dm(z) <\infty
$$
has the reproducing kernel, up to a normalization
constant (which will be determined below
for general Bergman spaces), 
\begin{equation}
\label{berg-kernel-h-p}
\det B(z, w)^{-1}
=
h(z, w)^{-p} 
\end{equation}
where $h(z, w)$ is an irreducible
polynomial holomorphic in $z$
and anti-holomorphic in $w$
and of maximal bi-degree $(r, r)$;
see e.g. \cite{fako, kora}.
Now by \cite[Corollary 3.15]{loos} for $z = \sum_{j=1}^r \lambda_j e_j$
$$ h(z,z) = \prod_{j=1}^r (1 - |\lambda_j|^2).$$
Note that this actually describes $h(z,z)$ for any $z \in \mathbb C^N$ as 
$$\mathbb C^N = \Ad(K)(\sum_{i=1}^r \R_{\geq 0} e_i).$$

The Bergman metric on $D$ 
at $z\in D$ is given by
$$ \langle v, w \rangle_z = \langle B(z, \li{z})^{-1} v, w \rangle_{\mathbb C^N} $$
for $v,w \in T_z D = \mathbb C^N$. 
By the transformation
property
(\ref{transformrule}) the Bergman metric is invariant under $G$. 
 We note that for $k \in K$ we have $B(k \cdot z, \li{k \cdot z}) = k B(z,\li{z}) k^{-1}$, and thus for
$$z = \lambda_1 e_1 + \dots \lambda_r e_r$$
we get that
$$\det(B(k \cdot z,\li{k \cdot z})^{-1}) = \det(B(z,\li{z})^{-1})=
 \prod_{j=1}^r (1 - |\lambda_j|^2)^{-p}.$$

The $G$-invariant Riemannian measure on $D=G/K$ is obtained  from the Bergman metric by
\begin{equation}
\label{bergmanmeasure}
   d \iota(z) =\det B(z, z)^{-1}dm(z)= h(z,z)^{-p} dm(z).
\end{equation}

\subsection{Bergman space
of vector-valued holomorphic functions}

Let $(V_{\Lambda}, \tau_{\Lambda}, K)$ 
be an irreducible unitary representation of $K$ of 
highest weight $\Lambda$. It
 can be extended to a rational representation of 
 $K^{\C}$
 on the space $V_{\Lambda}$.
 The $K$-unitary inner product on $V_{\Lambda}$ will be
 denoted by $\langle -,- \rangle_{\tau}$.

We now introduce the holomorphic discrete series.

\begin{definition}
\label{disseries}
Let $(V_{\Lambda},  \tau_\Lambda, 
K)$ be an irreducible representation of $K$ with highest weight $\Lambda$. 
Let $\Hc_{\Lambda}$ 
be  the Hilbert space of holomorphic
functions $f: D \rightarrow V_{\Lambda} $
with the norm square
\begin{equation}
\label{Berg-norm}
\Vert f\Vert_{\Hc_{\Lambda}}^2
:=\int_{D} \langle \tau(B(z,\li{z})^{-1})f(z), f(z) \rangle_{\tau} d \iota(z) < \infty.
\end{equation}
The holomorphic discrete series is
$(\Hc_{\Lambda}, 
\pi_{\Lambda}, G),$
with the  unitary representation \begin{equation}
(\pi_{\Lambda}(g) f ) (z) = \tau(J_{g^{-1}}(z)^{-1}) f(g^{-1} \cdot z),\end{equation}
provided $\Hc_{\Lambda}$ is non-trivial. 
\end{definition}

Indeed, the space $\Hc_{\Lambda} $
in Definition \ref{disseries} could be trivial. 
The Harish-Chandra condition give 
a characterization 
for $\Hc_{\Lambda} $; see e.g. \cite[Lemma 27, Paragraph 9]{hari}, \cite[equality (6)]{kora},  \cite[II, Theorem 6.5]{Wa}.

\begin{theorem}
\label{hccondition}
Let  $\Lambda$
be the highest weight of $(V_{\Lambda}, \tau, K)$
and let $\rho
=\frac 12\sum_{
\alpha \in \Delta^+}\alpha
$
be the half sum
of positive roots $\Delta^+$.
If
$$ (\Lambda + \rho)(h_1) < 0$$
then the Hilbert 
space $\Hc_{\Lambda} \neq \{0\}$
and defines a discrete series
of $G$.
\end{theorem}

The 
Hilbert space 
$\Hc_{\Lambda}$ has
reproducing kernel 
 $K_w(z)=K(z, w)=K_{\Lambda}(z, w)$ taking values in
$\mathrm{End}(V_{\Lambda})$,
holomorphic in $z$  and anti-holomorphic in $w$ 
such that for any $v \in V_{\Lambda},
f \in \Hc_{\Lambda}$, we have $K_w v \in \Hc_{\Lambda}$, and
$$ \langle f, K_w v \rangle_{\Hc_{\Lambda}} = \langle f(w), v \rangle_{\tau}.
$$
The kernel
$K$ can be computed
using the Bergman operator \cite[Paragraph 4]{kora}:
There is
a constant $C(\Lambda) > 0$, to be evaluated in Theorem \ref{constantthm},
such that
\begin{equation}
\label{kernelexpl}
 K(z,z) = C(\Lambda)
\tau(B(z,\li{z})) = 
C(\Lambda) \tau(J_g(0) J_g(0)^*)
\end{equation}
where $z = g \cdot 0$. It follows by holomorphicity in $z$ and anti-holomorphicity in $w$ that
$$ K(z,w) = C(\Lambda) \tau(B(z,\li{w})).$$
 Furthermore
$$K(g\cdot z, g \cdot w) = \tau(J_g(z)) K(z,w) \tau(J_g(w))^*.$$
From (\ref{kernelexpl}) 
we see that for any $z \in D$
$$
K(z,0) = C(\Lambda) I.$$
Thus for any $v \in V_{\Lambda}$ the constant function $v$ is in $\Hc_{\Lambda}$, as $v = C(\Lambda)^{-1} K_0 v \in \Hc_{\Lambda}$.

Furthermore, the space of $V_\Lambda$-valued polynomials
is dense in
$\Hc_\Lambda$,
and as a representation of $K$ it is
$\mathcal P\otimes V_\Lambda$ where
$
\mathcal P
$ is the space of scalar-valued polynomials;
see 
e.g. \cite{EHW}.

\section{The formal degree of the holomorphic discrete series}\label{for-deg}
The formal degree of the
discrete series
$(\Hc, \pi, G)$
of a semisimple Lie group $G$
is a proportionality constant
between the $|\langle  u, v\rangle|^2$ for $u, v\in \Hc$ 
and the $L^2(G)$-norm
square
of the matrix
coefficient
$\langle \pi(g)u, v\rangle$.
Harish-Chandra
\cite{hari} has computed
the formal degree up to a normalization constant.
We shall find  
the exact formula
for the formal degree under our normalization 
(\ref{Eucl})
above. The formal
degree will appear in the Wehrl inequality in the next Section.

\subsection{Definition of the formal degree}
Harish-Chandra \cite[Theorem 1]{hari} shows that for a holomorphic discrete series representation $\Hc_{\Lambda}$ and $f_1, f_2 \in \Hc_{\Lambda}$ there exists a positive number $d_{\Lambda}$,
called the formal degree 
of $\Hc_{\Lambda}$, 
such that
\begin{equation}
\label{formaldegreeort}
\int_G 
|\langle g \cdot f_1, f_2 \rangle_{\Hc_{\Lambda}}|^2 d g = d_{\Lambda}^{-1} \nm{f_1}^2 \nm{f_2}^2,
\end{equation}
where all the inner products
are in $\Hc_{\Lambda}$.
We now normalize  the Haar measure on $G$
so that
$$ \int_G f(g) dg = \int_D \left( \int_K f(x k) dk \right) d\iota(x),$$
where we realize $G$
as the set
$D\times K$
with the
invariant measure on $D=G/K$ 
from (\ref{bergmanmeasure}) and the Haar measure  $K$ is normalized so that
$\int_K dk =1$.

Harish-Chandra 
found a formula
for the formal degree up to
 some normalization of the Haar measure on $G$ \cite[Theorem 4]{hari}.
It is given 
by the following 
\begin{equation}
\label{HC-fml-dgr}
d_{\Lambda}^{\mathrm{H}}
:= 
(-1)^{\frac{\text{dim}G - \text{rank K}}2}
\prod_{\alpha \in \Delta^+}  \frac{\Lambda(h_{\alpha}) + \rho(h_{\alpha})}{\rho(h_{\alpha})},
\end{equation}
where $\rho = \frac{1}{2} \sum_{\alpha \in \Delta^+} \alpha$.
(Harish-Chandra's formula was 
the absolute of the above formula without the sign
$(-1)^{\frac{\text{dim}G - \text{rank K}}2}$, and
we take the sign with us to make it
a polynomial in $\Lambda$ and coincide
with the absolute value for discrete series.)

It follows that
there is a constant $c_G$ such that
\begin{equation}
\label{ourharidegree}
d_{\Lambda} = c_G \cdot d_{\Lambda}^{\mathrm{H}}.
\end{equation}
We shall find this constant by choosing scalar
representations $\tau$ of $K$ and by evaluating both degrees.

\subsection{Scalar
holomorphic discrete series}
This series
of representations
is very well understood; see e.g.
\cite{EHW, fako, Wa}.
Let 
$\lambda\in \mathbb Z_+$
be an integer and let
$\tau(k)   \coloneqq \det(\Ad(k)|_{\mathbb C^N})^{- \frac{\lambda}{p}}$, 
$k\in K$. 
Then up to a covering
of $G$ $\tau$ defines
a character of $K$,
and the covering will
have no effect on our
results as we have
fixed the integration 
of $K$ so that $\int_K dk =1$.
We see that for $H \in \h$ the scalar highest weight 
$\Lambda$ of $\tau$
is given by
$$ \Lambda(H) = \frac{d}{dt}|_{t=0} \tau(e^{tH}) = \frac{d}{dt}|_{t=0} e^{- t \frac{\lambda}{p} \Tr( \ad(H)|_{\mathbb C^N})} = - \frac{\lambda}{p} 2 \rho_n(H),$$
where $\rho_n = \frac{1}{2} \sum_{\alpha \in \Delta_n^+} \alpha$. Thus we get for $1 \leq j \leq r$ \cite[(1.4)]{kora}
$$ \Lambda(h_j) = - \lambda.$$
We shall identify the
weight $\Lambda$ with the
scalar $-\lambda$ and write
the corresponding $\tau_{\Lambda}$
as
$\tau_{-\lambda}$.
The condition
in Theorem \ref{hccondition} 
becomes
$\lambda >p-1$
; see \cite[Section 4]{kora}.
The Hibert space 
$\Hc_{\Lambda}$ is 
usually called the weighted
Bergman space with weight $\lambda -p>-1$.
The norm square
(\ref{Berg-norm})
is now given 
$$
\Vert f\Vert^2_{\Hc_{\Lambda}}
=\int_D 
|f(z)|^2
h(z, z)^{\lambda} d\iota(z)
=
\int_D |f(z)|^2 h(z, z)^{\lambda-p} dm(z),
$$
with $\tau_{- \lambda}(B(z, z))^{-1} =h(z, z)^\lambda$,
and the
representation $\pi_{\Lambda}$ becomes
$$
\pi_{\Lambda}(g) f(z)
=\det(J_{g^{-1}}(z))
^{\frac{\lambda}{p}}
f(g^{-1}z), \quad g\in G.
$$
The following
result follows
easily from
the definition and
the mean value
property of
holomorphic functions.
\begin{lemma}
\label{constantasint}
Let $\lambda >p-1$
and $\Lambda$ be as above.
For the representation $\tau(k) = \tau_{-\lambda}(k)$ the formal dimension $d_{\Lambda}$ of $\Hc_{\Lambda}$ is given by
$$d_{\Lambda}^{-1} = \int_G h(g \cdot 0, g \cdot 0)^{\lambda} d \iota(g) = \int_D h(z,z)^{\lambda - p} dm(z).$$
\end{lemma}

\begin{proof} We take
$f_1=f_2=1$ 
in Equation
(\ref{formaldegreeort}).
The LHS becomes
$$  \int_{G} | \langle \pi_{\Lambda}(g) 1, 1 \rangle_{\Hc_{\Lambda}} |^2 dg,$$
and the integrand
is by $K$-invariance
$$ | \langle \pi_{\Lambda}(g) 1, 1 \rangle_{\Hc_{\Lambda}} |^2 = |(\pi_{\Lambda}(g) 1) (0) |^2 \iota_{\lambda}(D)^2,$$
where 
$\iota_\lambda(D) 
=\int_D h(z, z)^\lambda d\iota(z).$
Moreover by (\ref{Bzzeq}),
and (\ref{berg-kernel-h-p}),
$$|(\pi_{\Lambda}(g) 1) (0)|^2 = |\det(J_{g^{-1}}(0))|^{\frac{2 \lambda}{p}} = |h(g^{-1} \cdot 0, g^{-1} \cdot 0)|^{\lambda}$$
so that
the LHS is
\begin{equation*}\begin{split}
&\iota_{\lambda}(D)^2 
\int_G h(g^{-1} \cdot 0, g^{-1} \cdot 0)^{\lambda} d \iota(g)
 = 
\iota_{\lambda}(D)^2 \int_G h(g \cdot 0, g \cdot 0)^{\lambda} d \iota(g)
\\
&=
\iota_{\lambda}(D)^2 
\int_D h(z, z)^{\lambda}d\iota(z),
\end{split}\end{equation*}
by our normalization
of $dg$. The RHS of
(\ref{formaldegreeort})
is
$d_\Lambda^{-1}\iota_{\lambda}(D)^2$, and
our claims follows.
\end{proof}

\subsection{Evaluation of the 
constant $c_G$}
We will use Lemma \ref{constantasint} to find the exact value of $d_{\Lambda}$ for the scalar representation $\tau_{-\lambda}$
and further
for general
discrete series.
First we need some notation \cite{fako}. Let
$$ \Gamma_{a}(\mathbf{s}) \coloneqq 
 \prod_{j=1}^r \Gamma(s_j - (j-1) \frac{a}{2})$$
 be Gindikin's Gamma function
 associated with the root
 multiplicity $a$ (without the factor $(2 \pi)^{\frac{n_1 - r}{2}}$), 
for  vectors $\mathbf{s} = (s_1, \dots, s_r)$ and $\Gamma_{a}(\lambda) \coloneqq \Gamma_{a}( (\lambda, \dots, \lambda))$. Thus
$$ \frac{\Gamma_{a}(\lambda - \frac{n}{r})}{\Gamma_{a}(\lambda)} = \prod_{j=1}^r \frac{\Gamma( \lambda - \frac{N}{r} - (j-1)\frac{a}{2} )}{\Gamma( \lambda - (j-1)\frac{a}{2})}.$$
A sketch for the 
evaluation of the integral $ d_{\Lambda}^{-1} $
was given \cite[Theorem 3.6]{fako};
we give a detailed proof 
by using
the known evaluation formula for the Selberg integral \cite{AAR}
as they are of importance for our main results.

\begin{proposition}
If $\tau = \tau_{- \lambda}$ with $\lambda > p - 1$ then we have
$$ d_{\Lambda}^{-1} 
= \int_D h(z)^{\lambda - p} d m(z) = \frac{\pi^N \Gamma_{a}(\lambda - \frac{N}{r})}{\Gamma_{a}(\lambda)}.$$
\end{proposition}

\begin{proof}
The first equality is Lemma \ref{constantasint}. 
We evaluate 
the integral by starting with the polar decomposition \cite[Chapter I, Theorem 5.17]{helgGGA} for $\mathbb C^N$,
\begin{equation*} \int_{\mathbb C^N} f(z) dm(z)
 = C \int_0^{\infty} \dots \int_0^{\infty} \int_K f(k \cdot (t_1 e_1 + \dots t_r e_r)) dk 2^r \prod_j t_j^{2b+1} \prod_{j<k} |t_j^2 - t_k^2|^a dt_1 \dots dt_r,
\end{equation*}
for some  constant $C$.
We calculate the exact value of this constant $C$. Let $f$ be the $K$-invariant
Gaussian function $f(z) = e^{-\nm{z}^2}$, then
\begin{equation*}\begin{split}
 \pi^N &= \int_{\mathbb C^N} e^{-\nm{z}^2} dm(z)
\\ &  = C \int_0^{\infty} \dots \int_0^{\infty}  e^{-(t_1^2 \dots + t_r^2)} 2^r \prod t_j^{2b + 1} \prod_{j<k} |t_j^2 - t_k^2|^a dt_1 \dots dt_r
\\ & = C \int_0^{\infty} \dots \int_0^{\infty}  e^{-(s_1 \dots + s_r)} \prod s_j^{b} \prod_{j<k} |s_j - s_k|^a ds_1 \dots ds_r.
\end{split}
\end{equation*}
From \cite[Corollary 8.2.2]{AAR} we find
\begin{equation*}\begin{split}
& \int_0^{\infty} \dots \int_0^{\infty}  e^{-(s_1 \dots + s_r)} \prod s_j^{b} \prod_{j<k} |s_j - s_k|^a ds_1 \dots ds_r
\\ & = \prod_j^r \frac{\Gamma(b+1+(j-1)\frac{a}{2}) \Gamma(1 + j \frac{a}{2})}{\Gamma(1 + \frac{a}{2})}.
\end{split}
\end{equation*}
Hence we have
\begin{equation}
\label{Const-C}
C = \pi^{N} \prod_j^r \frac{\Gamma(1 + \frac{a}{2})}{\Gamma(b+1+(j-1)\frac{a}{2}) \Gamma(1 + j \frac{a}{2})}.
\end{equation}
It  follows that
\begin{equation*}\begin{split}
& \int_{D} h(z)^{\lambda - p} d m(z)
\\ &  = C \int_0^{1} \dots \int_0^{1} h(t_1 e_1 + \dots t_r e_r)^{\lambda-p} 2^r \prod_j t_j^{2b+1} \prod_{j<k} |t_j^2 - t_k^2|^a dt_1 \dots dt_r
\\ & = C \int_0^{1} \dots \int_0^{1} \prod_j (1 - s_j)^{\lambda-p} \prod_j s_j^{b} \prod_{j<k} |s_j - s_k|^a ds_1 \dots ds_r.
\end{split}
\end{equation*}
This is a Selberg integral and
is evaluated by  \cite[Theorem 8.1.1]{AAR}
\begin{equation}
\begin{split}
\label{Selberg}
& \int_0^{1} \dots \int_0^{1} \prod_j (1 - s_j)^{\lambda-p} \prod_j s_j^{b} \prod_{j<k} |s_j - s_k|^a ds_1 \dots ds_r.
\\ & = \prod_{j=1}^r \frac{\Gamma(b+1+(j-1)\frac{a}{2}) \Gamma(\lambda - p + 1 + (j-1)\frac{a}{2}) \Gamma(1 + j \frac{a}{2})}{\Gamma(\lambda - p + b + 2 + (r + j - 2)\frac{a}{2})\Gamma(1 + \frac{a}{2})}.
\end{split}
\end{equation}
Hence
\begin{equation*}\begin{split}
 d_{\Lambda}^{-1} &= \int_{D} h(z)^{\lambda - p} d m(z)
\\ & = C \prod_{j=1}^r \frac{\Gamma(b+1+(j-1)\frac{a}{2}) \Gamma(\lambda - p + 1 + (j-1)\frac{a}{2}) \Gamma(1 + j \frac{a}{2})}{\Gamma(\lambda - p + b + 2 + (r + j - 2)\frac{a}{2})\Gamma(1 + \frac{a}{2})}
\\ & = \pi^N \prod_{j=1}^r \frac{\Gamma(\lambda - p + 1 + (j-1)\frac{a}{2})}{\Gamma(\lambda - p + b + 2 + (r + j - 2)\frac{a}{2}) }
\\ & = \pi^N\frac{\Gamma_{a}(\lambda - \frac{N}{r})}{\Gamma_{a}(\lambda)}.
\end{split}
\end{equation*}
\end{proof}

Now we can finally find the exact
value of the constant $c_G$.
Recall the Pochammer symbol
 $(x)_k= x(x+1)\cdots (x+k-1)$.

\begin{proposition}
\label{constanthcvsour}
With  our normalization of the Haar measure  the formal 
degree
 is given by
 $$
d_{\Lambda} = c_G d_{\Lambda}^{\mathrm{H}},$$
where $d_{\Lambda}^{\mathrm{H}}$ 
is given by  \cref{HC-fml-dgr} and
$$ c_G =\pi^{-N} r! \prod_{i=1}^{r-1}
(2+i)_i.$$
for $G = Sp(n, \mathbb R)$,
$$c_G = \pi^{-N} (m - \frac{1}{2}) (2m-2)!$$
for $G = SO_0(2,2m-1)$,
and
$$c_G = \pi^{-N} \prod_{j=1}^r \frac{\Gamma(\frac{N}{r} + (j-1)\frac{a}{2} + 1)}{\Gamma(1 + (j-1) \frac{a}{2})} = \pi^{-N} \prod_{j=1}^r (1 + (j-1)\frac{a}{2})_{\frac{N}{r}}$$
for all other irreducible Hermitian Lie groups.
\end{proposition}

\begin{proof}
The constant $c_G$
is independent of $\Lambda$
and can be found
by choosing the representations $\tau_{-\lambda}$ 
above. First we investigate $d_{\Lambda}^{\mathrm{H}}$
by evaluating 
$\Lambda(h_\alpha)$ on the co-roots $h_\alpha$.
From \cite[(1.4)]{kora} we see that for any of the strongly orthogonal co-roots $h_j$ we have
$$ \Lambda(h_j)
= - \frac{\lambda}{p} 2 \rho_n(h_j) = - \lambda.$$
In fact, using that
$$ \kf = \C Z \oplus [\kf, \kf]$$
is an orthogonal decomposition, and the fact that $\Lambda|_{[\kf, \kf]} = 0$, we get that 
$$ -\lambda = \Lambda(h_1) = \Lambda( \frac{ \langle h_{1}, Z \rangle}{\langle Z,Z \rangle} Z) = \frac{2 \gamma_1(Z)}{\langle \gamma_1, \gamma_1 \rangle \langle Z,Z \rangle} \Lambda(Z) = - \frac{2 i \Lambda(Z)}{\langle \gamma_1, \gamma_1 \rangle \langle Z,Z \rangle}.$$
Thus for any root $\alpha \in \Delta^+$ we get that
$$ \Lambda(h_{\alpha}) = \Lambda( \frac{ \langle h_{\alpha}, Z \rangle}{\langle Z,Z \rangle} Z) = \frac{2\alpha(Z) \Lambda(Z) }{\langle Z, Z \rangle \langle \alpha, \alpha \rangle} = - \frac{2i \Lambda(Z) }{\langle Z, Z \rangle \langle \alpha, \alpha \rangle} = - \lambda \frac{ \langle \gamma_1, \gamma_1 \rangle}{\langle \alpha, \alpha \rangle}.$$
By \cite{schli} there are short and long roots and the strongly orthogonal roots are always long. Say they are of length $l$, then the short roots are of length $\frac{l}{\sqrt{2}}$. Then
$$ \Lambda(h_{\alpha}) = - \lambda$$
if $\alpha$ is long, and
$$ \Lambda(h_{\alpha}) = - 2 \lambda$$
if $\alpha$ is short.

We compare 
$d_{\Lambda}^{
\mathrm{H}}$
with $d_{\Lambda}$
as polynomials of $\lambda$.
We divide $\g$ into three cases
depending on the 
multiplicity being $a=1$, $a>1$ odd, and  even. 
See e.g \cite{EHW}
for a list of 
all irreducible Hermitian Lie groups $G$.


Case 1: 
$G=Sp(r,\mathbb R)$. Here  $a=1$ and $\g^{\mathbb C}=\mathfrak{sp}(r, \mathbb C)$
has
$\Delta^+
=\{2\epsilon_j\}_{j=1}^m
\cup \{\epsilon_i\pm \epsilon_j
\}_{1\le i<j\le }$,
 the Harish-Chandra roots
are $\{2\epsilon_j\}_{j=1}^m$,
$\rho= \sum_{i=1}^r (r+1-i)\epsilon_i$,
and  $\Lambda =-\frac{\lambda}{2}(
\epsilon_1 +\cdots+
\epsilon_1).
$
We have
$$
d_{\Lambda}
= \pi^{-N}
\prod_{i=1}^r 
\frac{  \Gamma(\lambda -\frac 12 (i-1))
}
{
 \Gamma(\lambda -\frac 12 (r+i))
}
$$
is a polynomial of leading constant $ \pi^{-N}$.
The Harish-Chandra formal degree 
is
$$
d_{\Lambda}^{
\mathrm{H}} 
= \prod_{i=1}^r\frac{2 \lambda-(r+1-i)}{r+1-i}
\prod_{1\le i<j\le r}
\frac{\lambda-(r+1-\frac{i+j}2)}
{r+1-\frac{i+j}2}.
$$
Comparing this with $d_{\Lambda}$
we find
$$
c_G =\pi^{-N} r! \prod_{1\le i<j\le r}
(r+1-\frac{i+j}2)
= \pi^{-N} 2^{-\frac{(r-1)r}2} r!
\prod_{i=1}^{r-1}(2+i)_i.
$$

Case 2: 
$G=SO_0(2,2m-1)$. Here $r=2$,
$N=2m-1$, $a=2m-3$ is odd, and $\frac{N}{r}$ is not an integer. We get
\begin{equation*}
 d_{\Lambda} = \pi^{-N} 
\frac{\Gamma(\lambda) \Gamma(\lambda - \frac{2m-3}{2})}{\Gamma(\lambda - 1 - \frac{2m-3}{2}) \Gamma(\lambda - 2m + 2)} = \pi^{-N} (\lambda - m + \frac{1}{2})(\lambda - 2m + 2) \dots (\lambda-1),
\end{equation*}
The root system of 
$\g$ has
 $\Delta^+
=\{\epsilon_j\}_{j=1}^m
\cup \{\epsilon_i\pm \epsilon_j
\}_{1\le i<j\le }$
with the Harish-Chandra strongly
orthogonal roots being $\epsilon_1+\epsilon_2,
\epsilon_1-\epsilon_2$.
Now $\rho=\frac12 \sum_{i=1}^m (2m+1-2i)\epsilon_i$ and comparing the two polynomials we find
$$c_G =
\pi^{-N} (m- \frac{1}{2})(2m-2)!.$$

Case 3: The remaining cases.
All roots are of the same (long) length
as the Harish-Chandra roots, 
with $a$ being
even and  $N/r$  an integer \cite{EHW}.
We have
\begin{equation*}
\begin{split}
d_{\Lambda} &= \pi^{-N}
\frac{\Gamma_{a}(\lambda)}{\Gamma_{a}(\lambda - \frac{N}{r})} = \pi^{-N}
\prod_{j=1}^r \frac{\Gamma(\lambda - (j-1)\frac{a}{2})}{\Gamma(\lambda - \frac{N}{r} - (j-1)\frac{a}{2})}
\\
&= \pi^{-N}
\frac{\Gamma_{a}(\lambda)}{\Gamma_{a}(\lambda - \frac{N}{r})} = \pi^{-N} \prod_{j=1}^r \frac{\Gamma(\lambda - (j-1)\frac{a}{2})}{\Gamma(\lambda - \frac{N}{r} - (j-1)\frac{a}{2})}
\\ 
& = 
\pi^{-N}\prod_{j=1}^r 
(\lambda - (j-1) \frac{a}{2} - \frac{N}{r})_{\frac{N}r}
\end{split}
\end{equation*}
and as a polynomial of $\lambda$ its zeros
are all given and it has leading coefficient $\pi^{-N}$.
Also,
$$ d_{\Lambda}^{
\mathrm{H}} 
= \prod_{\alpha \in \Delta^+}  \frac{\Lambda(h_{\alpha}) + \rho(h_{\alpha})}{\rho(h_{\alpha})} = \prod_{\alpha \in \Delta^+} \frac{- \lambda + \rho(h_{\alpha})}{\rho(h_{\alpha})}$$
is a polynomial with the coefficient of the leading term being $\prod_{\alpha \in \Delta^+} \rho(h_{\alpha})^{-1}$
and zeros $\rho(h_\alpha)$. 
It follows that the  product of zeros is
$$ \prod_{j=1}^r \frac{ \Gamma(\frac{N}{r} + (j-1)\frac{a}{2} + 1)}{\Gamma(1 + (j-1)\frac{a}{2})}.$$
Consequently
$$ c_G = \pi^{-N} \prod_{j=1}^r \frac{\Gamma(\frac{N}{r} + (j-1)\frac{a}{2} + 1)}{\Gamma(1 + (j-1) \frac{a}{2})} 
= \pi^{-N} \prod_{j=1}^r (1 + (j-1)\frac{a}{2})_{\frac{N}{r}}.$$
This finishes the proof.
\end{proof}

As a corollary we can find $\langle v,v \rangle_{\Hc_{\Lambda}}$ for $v \in V_{\Lambda}$, the constant $C(\Lambda)$ and a precise
formula for the
reproducing kernel.

\begin{theorem}
\label{constantthm}
Let $\Lambda$ be as above. Then for any unit vector $v \in V_{\Lambda}$
$$ \langle v, v \rangle_{\Hc_{\Lambda}} = d_{\Lambda}^{-1}.$$
Furthermore, the reproducing kernel for
the space $\Hc_{\Lambda}$
is given by
$$
K(z, w) = d_{\Lambda} \tau(B(z, \li{w})).
$$
\end{theorem}

\begin{proof}
From (\ref{kernelexpl}) we have
$K(z,w) = C(\Lambda) \tau(B(z,\li{w}))$,  in particular
$K(z,0) = C(\Lambda) I$
and
$$ 
\langle f(0), v \rangle_{\tau}
=C(\Lambda) 
\langle f, v \rangle_{\Hc_{\Lambda}}.
$$
It follows by the reproducing kernel formula that  for any $v,w \in V_{\Lambda}
\subset \Hc_{\Lambda}$, 
\begin{equation}
\label{Schur-1}
\begin{split}
 d_{\Lambda}^{-1} |\langle v, w \rangle_{\tau}|^2 & =
|C(\Lambda)|^{2}
d_{\Lambda}^{-1} |\langle v, w \rangle_{\Hc_{\Lambda}}|^2
 =
|C(\Lambda)|^{2}
\int_G |\langle \pi_{\Lambda}(g) v, w \rangle_{\Hc_{\Lambda}} |^2
 dg
\\ & =  
\int_G | \langle \pi_{\Lambda}(g) v (0), w \rangle_{\tau}|^2 dg
 =  \int_G 
|\langle \tau(J_{g^{-1}}(0))^{-1} v, w \rangle_{\tau}|^2 dg
\\
& =  
\int_G |\langle \tau(J_g(0))^{-1} v, w \rangle_{\tau}|^2 dg.
\end{split}\end{equation}

Let $v$ 
be a unit vector
 and 
$\{ v_i \}_{i=1}^d$
an orthonormal basis
of $V_{\Lambda}$, where $d = \dim(V_{\Lambda})$.  We compare
$\langle v, v \rangle_{\Hc_{\Lambda}}$  with $L^2$-
square norm of
the corresponding matrix coefficients:
\begin{equation*}\begin{split}
 d 
\langle v, v \rangle_{\Hc_{\Lambda}} &= d C(\Lambda)^{-1} = C(\Lambda)^{-1} \sum_{i=1}^d \langle v_i, v_i \rangle_{\tau} = \sum_{i=1}^d \langle v_i, v_i \rangle_{\Hc_{\Lambda}}
\\ & = \int_D \sum_{i=1}^d \langle \tau(B(z,\li{z}))^{-1} v_i, v_i \rangle_{\tau} d \iota(z) = \int_D \Tr(\tau(B(z,\li{z}))^{-1}) d \iota(z)
\\ & = \int_G \Tr(\tau(J_g(0) J_g(0)^*)^{-1}) dg = \sum_{i=1}^d \int_G \langle \tau(J_g(0) J_g(0)^*)^{-1} v_i, v_i \rangle_{\tau} dg
\\ & = \sum_{i,j=1}^d \int_G |\langle \tau(J_g(0))^{-1} v_i, v_j \rangle_{\tau}|^2 dg ={ 
d_{\Lambda}^{-1} \sum_{i,j=1}^d 
| \langle v_i, v_j \rangle_{\tau}|^2}
= d d_{\Lambda}^{-1},
\end{split}\end{equation*}
where the penultimate equality is by
(\ref{Schur-1}).
Hence we get that for any unit vector $v$
$$ \langle v, v \rangle_{\Hc_{\Lambda}} = d_{\Lambda}^{-1}.$$
We also obtain that the constant $C(\Lambda)$ is given by $C(\Lambda) = d_{\Lambda}$.
\end{proof}

\begin{remark} We 
note that as a consequence we obtain
the following integral evaluation
\begin{equation*}
\begin{split}
& \quad d_\Lambda
\int_D \Tr\left(\tau_{\Lambda}(B(z, z)^{-1})\right)
d\iota(z)
\\
&= C d_\Lambda 
\int_{[0, 1]^r}
\Tr (\tau_{\Lambda}
\bigg(B(
\sum_{j=1}^r t_j e_j,
\sum_{j=1}^r t_j e_j,
)^{-1} \bigg))
2^r \prod_j t_j^{2b+1} \prod_{j<k} |t_j^2 - t_k^2|^a dt_1 \dots dt_r\\
&=\dim(V_{\Lambda}),
\end{split}
\end{equation*}
for any unit vector $v$, where $C$ is the constant
(\ref{Const-C}). This might be
viewed as an generalization of the Selberg integral
(\ref{Selberg});
in other words,
the result
is a consequence
of the Selberg integral
evaluation and the Harish-Chandra formula for formal degree.
\end{remark}
\section{Wehrl inequality
for holomorphic discrete series}\label{Wehrl-ineq}
We  prove 
our main results
on Wehrl-type inequalities.
We keep the previous notation.
The tensor product below
$\mathcal H_1\otimes 
\mathcal H_2$
of two Hilbert spaces of
holomorphic functions
on $D$ 
will be realized
as a space
of holomorphic
functions $F(z, w)$ in two variables.

\subsection{Tensor products
of holomorphic discrete series
and intertwining operators
}
 We recall some
known
results
on tensor product of
holomorphic discrete series
representations
\cite{repka}.

\begin{proposition}
\label{repka-pro}
Let 
$ (\Hc_{\Lambda}, \pi_{\Lambda}, G) $
and $
(\Hc_{\Lambda'},
\pi_{\Lambda'}, G) 
$
be two holomorphic
discrete series representations
of highest weights 
$\Lambda$ and ${\Lambda'}$.
Then 
$\Hc_{\Lambda} \otimes \Hc_{\Lambda'}$
is a direct sum of representations of the form $\pi_{\Lambda''}$ with finite multiplicities. The corresponding highest weights $\Lambda''$  are of the form
$$ \Lambda'' = \Lambda_0 - (m_1 \alpha_1 + \dots + m_q \alpha_q),$$
where $\Lambda_0$ is a weight of $V_{\Lambda} \otimes V_{\Lambda'}$,  $m_i$ are nonnegative integers and the $\alpha_i \in \Delta_n^+$. In particular, there is an irreducible leading component
$$
\Hc_{\Lambda + \Lambda'} \subseteq \Hc_{\Lambda} \otimes \Hc_{\Lambda'}
$$
which is obtained by the intertwining
map
\begin{equation}
\label{J0}
J_0(F)(z)= P_{\Lambda +\Lambda'} F(z, z).
\end{equation}
Here $P_{\Lambda + \Lambda'}: V_{\Lambda} \otimes V_{\Lambda'} \rightarrow V_{\Lambda + \Lambda'}
\subseteq
V_{\Lambda} \otimes V_{\Lambda'} 
$ is the orthogonal projection. Moreover $\Hc_{\Lambda + \Lambda'} $
appears in $
\Hc_{\Lambda} \otimes \Hc_{\Lambda'}$
with multiplicity one.

For any 
irreducible
subrepresentation
of $V_\Lambda\otimes
V_{\Lambda'}$
of $K$
with highest weight $\Lambda_0$
and the corresponding
projection $P_{\Lambda_0}: V_{\Lambda} \otimes V_{\Lambda'} \rightarrow V_{\Lambda_0}$ the map
$$F(z,w) \mapsto P_{\Lambda_0} F(z,z)$$
is an intertwining map onto an irreducible component of $\Hc_{\Lambda} \otimes \Hc_{\Lambda'}$
of highest weight
$\Lambda_0$.
\end{proposition}

We now
find the exact constant $C_{\Lambda,\Lambda'}$ such that $C_{\Lambda, \Lambda'} J_0$ is a partial isometry.

\begin{proposition}
\label{propconstant}
Let $\Hc_{\Lambda}$,
$\Hc_{\Lambda'}$
and $J_0$ be as in Proposition
\ref{repka-pro}.
Then $C_{\Lambda, \Lambda'} J_0
$ is a partial isometry, where
$$C_{\Lambda, \Lambda'}^{-2} = c_G |\prod_{\alpha \in \Delta^+} \frac{ (\Lambda(h_{\alpha}) + \rho(h_{\alpha}))(\Lambda'(h_{\alpha}) + \rho(h_{\alpha}))}{((\Lambda + \Lambda')(h_{\alpha}) + \rho(h_{\alpha})) \rho(h_{\alpha})}|.$$
\end{proposition}

\begin{proof}
The constant holomorphic function $v$  is in $ \Hc_{\Lambda}$, for any $v \in V_{\Lambda}$ \cite{kora}, and if $v_{\Lambda} \in V_{\Lambda}$ and $v_{\Lambda'} \in V_{\Lambda'}$ are highest weight vectors of unit length then so is 
$P_{\Lambda + \Lambda'}(v_{\Lambda} \otimes v_{\Lambda'})
=v_{\Lambda} \otimes v_{\Lambda'}
$. Hence
$$ \nm{J_0(v_{\Lambda} \otimes v_{\Lambda'})}_{\Hc_{\Lambda + \Lambda'}} = \nm{P_{\Lambda + \Lambda'}(v_{\Lambda} \otimes v_{\Lambda'})}_{\Hc_{\Lambda + \Lambda'}} = C_{\Lambda, \Lambda'}^{-1} \nm{v_{\Lambda}}_{\Hc_{\Lambda}} \nm{v_{\Lambda'}}_{\Hc_{\Lambda}}.$$
 Using Theorem \ref{constantthm} we see that
\begin{equation*}\begin{split}
 C_{\Lambda, \Lambda'}^{-2} &= \frac{\nm{P_{\Lambda + \Lambda'}(v_{\Lambda} \otimes v_{\Lambda'})}^2}{\nm{v_{\Lambda}}^2 \cdot \nm{v_{\Lambda'}}^2} = \frac{d_{\Lambda} d_{\Lambda'}}{d_{\Lambda + \Lambda'}} = c_G \frac{d_{\Lambda}^{\mathrm{H}} d_{\Lambda'}^{\mathrm{H}}}{d_{\Lambda+\Lambda'}^{\mathrm{H}}}
\\ & = c_G |\prod_{\alpha \in \Delta^+} \frac{(\Lambda(h_{\alpha}) + \rho(h_{\alpha}) )(\Lambda'(h_{\alpha}) + \rho(h_{\alpha}))}{((\Lambda + \Lambda')(h_{\alpha}) + \rho(h_{\alpha}))(\rho(h_{\alpha}))}.
\end{split}\end{equation*}
\end{proof}

\subsection{Wehrl inequality}

We write $Q_0 = C_{\Lambda, \Lambda'} J_0$. 
We now prove our main result
on the Wehrl-type inequality.

\begin{theorem}
\label{wherlineq}
The following Wehrl inequality
holds for
$f
\in \Hc_{\Lambda}
$ and integers $n\ge 2$,
\begin{equation*}\begin{split}
& \int_{D} \langle P_{n\Lambda }( (\tau_{\Lambda}(B(z,\li{z})^{-1}) f(z))^{\otimes n} ), P_{ n\Lambda }( f(z)^{\otimes n}) \rangle_{\tau_{n \Lambda}} d \iota(z)
\\ & \leq 
c_G^{n-1} \frac{(d_{\Lambda}^{\mathrm{H}})^n}{d_{n \Lambda}^{\mathrm{H}}} \left( \int_{D} \langle \tau_{\Lambda}(B(z,\li{z})^{-1}) f(z), f(z) \rangle_{\tau_{\Lambda}} d \iota(z) \right)^n.
\end{split}
\end{equation*}
The equality holds if and only if $f = K_u \tau(k) v_{\Lambda}$ for some $u \in D$, $k \in K$ and $v_{\Lambda}$ a highest weight vector in $V_{\Lambda}$.
\end{theorem}

\begin{proof}
We have that
$$ \Hc_{n \Lambda} \subseteq \Hc_{\Lambda}^{\otimes n},$$
and it appears with multiplicity one by Proposition \ref{repka-pro}. Now let
$$P_{n\Lambda}: V_{\Lambda}^{\otimes n} \rightarrow V_{n \Lambda}$$
be the projection. The operator
$$ J_0 : \Hc_{\Lambda}^{\otimes n} \rightarrow \Hc_{n \Lambda}$$
defined by
$$J_0( f_1 \otimes \dots \otimes f_n)(z) \coloneqq P_{n \Lambda}(f_1(z) \otimes \dots \otimes f_n(z)),$$
is then an 
intertwining
map onto $\Hc_{n \Lambda}$; this is  Proposition \ref{propconstant} applied multiple times. 
Furthermore, by Proposition \ref{propconstant},
$$Q_0 = C_{\Lambda,n} J_0,
\quad
C_{\Lambda,n}^2 = c_G^{-n+1} \frac{
d_{n \Lambda}^{\mathrm{H}}}{(d_{\Lambda}^{\mathrm{H}})^n}
$$
is a partial isometry.
Applying this to the 
element $f^{\otimes n}$
for 
$f\in \Hc_{\Lambda}$ we get
$$ \nm{P_{n\Lambda }( f(z)^{\otimes n}
)
}_{n \Lambda}^2
\leq c_G^{n-1} \frac{(d_{\Lambda}^{
\mathrm{H}})^n}
{d_{n \Lambda}^{\mathrm{H}}} \nm{f^{\otimes n}}^{2} = c_G^{n-1} \frac{(d_{\Lambda}^{\mathrm{H}})^n}
{
d_{n \Lambda}^{\mathrm{H}}}
\nm{f}^{2n},$$
or more explicitly
\begin{equation*}\begin{split}
& \int_{D} \langle P_{n\Lambda }( (\tau_{\Lambda}(B(z,\li{z})^{-1} f(z))^{\otimes n}), P_{n \Lambda}( f(z)^{\otimes n}) \rangle_{\tau_{n \Lambda}} d \iota(z)
\\ & \leq c_G^{n-1} \frac{
(d_{\Lambda}^{\mathrm{H}})^n}
{
d_{n \Lambda}^{\mathrm{H}}} \left( \int_{D} \langle \langle \tau_{\Lambda}(B(z,\li{z})^{-1}) f(z), f(z) \rangle_{\tau_{\Lambda}} d \iota(z) \right)^n,
\end{split}\end{equation*}
proving  the inequality.

We prove the rest of
our Theorem for $n=2$,
and the same arguments are valid for  general $n$.
Note that by Proposition \ref{repka-pro} we can write
$$f \otimes f = \bigoplus f_{\Lambda''},$$
where $f_{\Lambda''} \in m_{\Lambda''} \Hc_{\Lambda''}$ and $f_{2 \Lambda} = Q_0(f \otimes f)$. Now the inequality is an equality if and only if
$$f_{\Lambda''} \ne 0 \Leftrightarrow \Lambda'' = 2 \Lambda.$$
This holds if and only if
\begin{equation}
\label{projnec}
f = Q_0^* Q_0 (f).
\end{equation}
This is clearly true if $f = K_u \tau(k) v_{\Lambda}$, because then if $u = g \cdot 0$
\begin{equation*}\begin{split}
 f(z)& = K(z,u) \tau(k) v_{\Lambda} = \tau(J_g(g^{-1} z) K(g^{-1} \cdot z, 0) \tau(J_g(0))^* \tau(k) v_{\Lambda}
\\ & = d_{\Lambda} \tau(J_{g^{-1}}(z))^{-1} \tau(J_g(0))^* \tau(k) v_{\Lambda} = d_{\Lambda} \pi_{\Lambda}(g) \left( \tau(J_g(0))^* \tau(k) v_{\Lambda} \right) (z).
\end{split}\end{equation*}
The identity (\ref{projnec}) then follows by $G$-invariance of $Q_0$ and the fact that the vector $\tau(J_g(0))^* \tau(k) v_{\Lambda}$ is a translate of a highest weight vector.

Now suppose $f\in \mathcal H_{\Lambda}$
is such that the equality (\ref{projnec}) holds.
By replacing $f$
by $\pi_\Lambda(g) f$ for some $g \in G$ we may assume
that $f(0) \ne 0$
is a unit vector (note $f$ is a 
reproducing kernel if and only if $\pi_{\Lambda}(g) f$ is). We prove first
that $f(z)=K(z, u)v$
for some $u \in D$
and  $v \in V_{\Lambda}$,
and then that 
the vector $v$ has
to be a translate $\tau(k) v_{\Lambda}$ of the highest weight vector $v_{\Lambda} \in V_{\Lambda}$.

To prove that 
$f(z)=K(z, u)v$
for some $u$ and $v$ we use the same
 idea as in \cite{Z-24}. 
 Let $z=(z_i)$
 be the coordinates
 of $z\in \mathbb C^N$ 
under some
orthonormal basis. We consider the Toeplitz
operator
$T_{i}$
by coordinate functions 
$T_i f(z)= z_i f(z)$
on the space $\mathcal H_{\Lambda}$.
First of all the operators $T_i$
are bounded  on $\mathcal H_{\Lambda}$; indeed
\begin{equation*}
\begin{split}
\Vert T_i f\Vert^2
&=\int_{D} \langle \tau(B(z,\li{z})^{-1})z_if(z), z_if(z) \rangle_{\tau} d \iota(z) 
=\int_{D} |z_i|^2\langle \tau(B(z,\li{z})^{-1})f(z), f(z) \rangle_{\tau} d \iota(z) \\
&\le     \int_{D} \Vert z\Vert^2\langle \tau(B(z,\li{z})^{-1})f(z), f(z) \rangle_{\tau} 
d\iota(z) \le C_0 \Vert f\Vert^2
\end{split}
\end{equation*}
since $D$ is bounded.
Write
$T_{i, 1} F(z, w)= z_i F(z, w)$
and $T_{i, 2} F(z, w)= w_i F(z, w)$
on the space $\mathcal \Hc_{\Lambda}
\mathcal \otimes \Hc_{\Lambda}$.
From the definition of $Q_0$ we see that for any $g \in \Hc_{\Lambda} \otimes \Hc_{\Lambda}$
$$Q_0( (T_{i,1} - T_{i,2}) g ) = Q_0((z_i - w_i)g) = 0.$$
Thus
\begin{equation*}\begin{split}
& \langle f \otimes f, (T_{i, 1} - 
T_{i, 2})g\rangle_{\Hc_{\Lambda} \otimes \Hc_{\Lambda}} = \langle Q_0^* Q_0 (f \otimes f), (T_{i, 1} - 
T_{i, 2}) g \rangle_{\Hc_{\Lambda} \otimes \Hc_{\Lambda}}
\\ & = \langle Q_0(f \otimes f), Q_0( (T_{i, 1} - 
T_{i, 2}) g ) \rangle_{\Hc_{2\Lambda}} = 0.
\end{split}\end{equation*}
Therefore $
(T_{i, 1} - 
T_{i, 2})^\ast (f \otimes f)=0$, 
which is the same as
$$
(T_{i, 1}^* f) \otimes f = 
f \otimes (T_{i, 2}^* f).$$
This implies that there is a $u_i \in \C$ such that
$$T_{z_i}^* f = u_i f.$$
We write $u = (u_1, \dots, u_N)$. This then implies that for any polynomial $p$ (where $\li{p}$ is the polynomial where the coefficients are the complex conjugates of the original one) in $D$ and $v =f(0)\in V_{\Lambda}$
\begin{equation*}\begin{split}
 \langle p v, f \rangle_{\Hc_{\Lambda}} &= 
\langle  p(T_{1}, \dots, T_{n}) v, f \rangle_{\Hc_{\Lambda}} = \langle v,\li{p}(T_{1}^*, \dots, T_{n}^*) f \rangle_{\Hc_{\Lambda}} = \langle v, \li{p}(u) f \rangle_{\Hc_{\Lambda}}
\\ & =  d_{\Lambda} \langle p(\li{u}) v, f(0) \rangle_{\tau_{\Lambda}}.
\end{split}
\end{equation*}
Thus for any $V_{\Lambda}$-valued polynomial $p$ 
$$
\langle p, f \rangle_{\Hc_{\Lambda}} 
= d_{\Lambda} \langle p(\li{u}), f(0) \rangle_{\tau_{\Lambda}}.$$
That is, $p \to \langle p(\li{u}), f(0) \rangle$ is
a bounded evaluation and 
so by Lemma \ref{evbdd} $\li{u} \in D$,
Furthermore, $f=K(z, \li{u})v$
for some $v \in V_{\Lambda}$, where $v = d_{\Lambda}^{-1} f(0)$.

Now we prove
$f(0) = \tau(k) v_{\Lambda}$ for $v_{\Lambda}$ a highest weight vector and 
for some $k \in K$. By Proposition \ref{repka-pro} we see that for any irreducible representation $V_{\Lambda_0} \subseteq V_{\Lambda} \otimes V_{\Lambda}$
the map
$$F(z,w) \mapsto P_{\Lambda_0}(F(z,z))$$
is an intertwining map $\Hc_{\Lambda} \otimes \Hc_{\Lambda} \rightarrow \Hc_{\Lambda_0}$. Thus if $\Lambda_0 \neq 2 \Lambda$ we have
$$ P_{\Lambda_0}(f(z) \otimes f(z)) = 0, \quad z\in D.$$
In particular 
$$ P_{\Lambda_0}(f(0) \otimes f(0)) = 0,$$
for $\Lambda_0 \neq 2 \Lambda$. This
reduces to  a
condition for tensor product
decomposition of finite-dimensional representations, 
and  by Lemma \ref{wehrlcptprop} we
obtain that
$f(0) = \tau(k) v_{\Lambda}$ for $v_{\Lambda}$ a highest weight vector.
\end{proof}

We reformulate the theorem as an $L^2(G)-L^p(G)$-estimate
for matrix coefficients.

\begin{corollary}
\label{wehrlineqcor}
Let $\Hc_{\Lambda}$
be as above. Then we
have
the following $L^2-L^{2n}$-estimates
$$ \int_G 
|\langle\pi(g)f, v_\Lambda\rangle_{\Hc_{\Lambda}}|^{2n} dg
\leq c_G^{n-1} \frac{(d_{\Lambda}^{\mathrm{H}} )^{n}}{d_{n \Lambda}^{\mathrm{H}}} 
\left( \int_G | \langle \pi(g) f, v_{\Lambda} \rangle_{\Hc_{\Lambda}} |^2 dg \right)^n,$$
and equality holds
if and only if $f$
is as in Theorem \ref{wherlineq} above.
\end{corollary}

\begin{proof} We realize
$V_{n\Lambda}$
as the leading component
in the 
tensor product
$V_{\Lambda}^{\otimes n} 
$
as above, with 
$v_{\Lambda}^{\otimes n} 
\in
V_{n\Lambda}\subseteq
V_{\Lambda}^{\otimes n}$.
By the proof of Theorem \ref{wherlineq} the projection
$$Q_0:
\Hc_{\Lambda}^{\otimes n}
\to \Hc_{\Lambda}^{\otimes n}
$$
is given by
$$Q_0(f^{\otimes n})(z) = c_G^{-n+1} \frac{
d_{n \Lambda}^{\mathrm{H}}}{(d_{\Lambda}^{\mathrm{H}})^n}
P_{n \Lambda}
f(z)^{\otimes n}.
$$
As $v_{\Lambda}^{\otimes n} \in \Hc_{n \Lambda} \subseteq \Hc_{\Lambda}^{\otimes n}$ the $L^{2n}$-norm
can be written, using
 (\ref{formaldegreeort}), as
\begin{equation*}
\begin{split}
\int_G 
|\langle\pi_{\Lambda}(g)f, v_\Lambda\rangle_{\Hc_{\Lambda}}|^{2n} dg
&=
\int_G 
|\langle\pi_{\Lambda}^{\otimes n}(g)f^{\otimes n}, v_\Lambda^{\otimes n}
\rangle_{\Hc_{\Lambda}}
|^{2} dg
\\ & = \int_G 
|
\langle
\pi_{n \Lambda}(g)
Q_0
f^{\otimes n}, Q_0(v_{\Lambda}^{\otimes n})
\rangle_{\Hc_{n \Lambda}}|^{2}dg
\\
&= d_{n\Lambda}^{-1}
\Vert 
Q_0
(f^{\otimes n})
\Vert_{\Hc_{n \Lambda}}^2
\Vert 
Q_0(v_{\Lambda}^{\otimes n})
\Vert_{\Hc_{n \Lambda}}^2,
\end{split}
\end{equation*}
with $$ \int_G 
|
\langle
\pi_{\Lambda}(g)
f, v_{\Lambda}
\rangle_{\Hc_{\Lambda}}|^{2}dg = d_{\Lambda}^{-1}
\Vert 
f
\Vert_{\Hc_{\Lambda}}^2
\Vert 
v_{\Lambda}
\Vert_{\Hc_{\Lambda}}^2$$
for $n=1$.
Now by Theorem \ref{wherlineq} we get
\begin{equation*}\begin{split}
& d_{n\Lambda}^{-1}
\Vert 
Q_0
(f^{\otimes n})
\Vert_{\Hc_{n \Lambda}}^2
\Vert 
Q_0(v_{\Lambda}^{\otimes n})
\Vert_{\Hc_{n \Lambda}}^2
\\ & = \left( c_G^{-n+1} \frac{
d_{n \Lambda}^{\mathrm{H}}}{(d_{\Lambda}^{\mathrm{H}})^n} \right)^2 d_{n\Lambda}^{-1}
\Vert 
P_{n\Lambda}
f^{\otimes n}
\Vert_{\Hc_{n \Lambda}}^2
\Vert 
P_{n \Lambda} (v_{\Lambda}^{\otimes n})
\Vert_{\Hc_{n \Lambda}}^2
\\ & \leq c_G^{-n+1} \frac{
d_{n \Lambda}^{\mathrm{H}}}{(d_{\Lambda}^{\mathrm{H}})^n} d_{n \Lambda}^{-1} 
\nm{f}_{\Hc_{\Lambda}}^{2n}
\Vert 
P_{n \Lambda} (v_{\Lambda}^{\otimes n})
\Vert_{\Hc_{n \Lambda}}^2
\\ & = \frac{1}{d_{\Lambda}^n} \nm{f}_{\Hc_{\Lambda}}^{2n}
\nm{P_{n \Lambda} (v_{\Lambda}^{\otimes n})}^2_{\Hc_{n \Lambda}}
\\ & = \frac{ \nm{P_{n \Lambda} (v_{\Lambda}^{\otimes n}) }_{\Hc_{n \Lambda}}^2}{ \nm{v_{\Lambda}}_{\Hc_{\Lambda}}^{2n}} \left( \int_G | \langle \pi_{\Lambda}(g) f, v_{\Lambda} \rangle_{\Hc_{\Lambda}} |^2 dg \right)^n,
\end{split}\end{equation*}
with equality if and only if $f = K_u \tau(k) v_{\Lambda}$ for some $u \in D$, $k \in K$ and $v_{\Lambda}$ a highest weight vector in $V_{\Lambda}$. We also know by Theorem \ref{constantthm}
$$\nm{P_{n \Lambda} (v_{\Lambda}^{\otimes n} )}_{\Hc_{n \Lambda}}^2 = d_{n \Lambda}^{-1},\quad
\nm{v_{\Lambda}}_{\Hc_{\Lambda}}^{2n} = d_{\Lambda}^{-n}.$$
We conclude that
$$ \int_G 
|\langle\pi(g)f, v_\Lambda\rangle_{\Hc_{\Lambda}}|^{2n} dg \leq \frac{d_{\Lambda}^{n}}{d_{n \Lambda}} \left( \int_G | \langle \pi(g) f, v_{\Lambda} \rangle_{\Hc_{\Lambda}} |^2 dg \right)^n,$$
with the constant $ \frac{d_{\Lambda}^{n}}{d_{n \Lambda}} = c_G^{n-1} \frac{(d_{\Lambda}^{\mathrm{H}} )^{n}}{d_{n \Lambda}^{\mathrm{H}}}$. This completes the proof.
\end{proof}

\begin{remark}
For the unit disc $D=SU(1,1)/U(1)$ a general inequality is proved in \cite{Fr, ku} with the $p$-norm replaced by any positive convex function. A  challenging 
problem would be to find optimal $L^2-L^p$
estimates for scalar holomorphic discrete series.
\end{remark}

\section{An improved Wehrl
inequality for
the unit disc}\label{improved-W}
\subsection{Irreducible decomposition
of tensor product
discrete series
of $SU(1, 1)$ and differential intertwining operators}

In this section we prove an improved
$L^2-L^{p}$ Wehrl inequality
for the holomorphic discrete series of $SU(1,1)$,
with $p=2n$  an even integer.
For the Fock space $\mathcal F(\mathbb C)$
or equivalently
the $L^2(\mathbb R)$-space
as representation
space of the Heisenberg
group $\mathbb R\rtimes \mathbb C$
an improved Wehrl-type inequality
(for any convex function instead of the $p$-norm)
was recently obtained in \cite{FNT}.
Our result here might provide a method
for obtaining a
more precise 
remainder term
for the improved
$L^2-L^p$-
Wehrl inequalites
for the Heisenberg group
and $SU(1, 1)$.

Let 
$\Hc_{\nu}$ be the weighted Bergman space
of holomorphic functions $f$ on
the unit disk $D\subset \mathbb C$ such that
$$ \nm{f}_{\nu,2}^2 \coloneqq (\nu - 1) \int_{\D} |f(z)|^2(1 - |z|^2)^{\nu} \frac{dm(z)}{\pi (1 - |z|^2)^2} <\infty
$$
where $dm(z)$ as above is the Lebesgue measure.
We also write
$$ \nm{f}_{\nu,p}^p \coloneqq (\nu - 1) \int_{\D} |f(z)|^p(1 - |z|^2)^{\frac{p \nu}{2}} \frac{dm(z)}{\pi (1 - |z|^2)^2}.$$
Note that if $\nu$ is an integer then $\Hc_{\nu}$ is the holomorphic discrete series representation for the representation $\tau_{-\nu}$ of $U(1) \subseteq SU(1,1)$
$$\tau_{\nu}( \begin{pmatrix} e^{i \theta} & 0 \\ 0 & e^{- i \theta} \end{pmatrix}) = e^{-i \nu \theta}.$$
The tensor product
of  holomorphic discrete series of $SU(1,1)$ has 
a decomposition \cite{repkasl2},
$$ \Hc_{\mu} \otimes \Hc_{\nu} \cong \bigoplus_{k=0}^{\infty} \Hc_{\mu + \nu + 2k},$$
where we normalize the inner product on all holomorpic discrete series so that $\langle 1,1 \rangle_{\mu+\nu +2k} = 1$.
Then we have partial isometries
$$ Q_{k}^{\mu, \nu}: 
\Hc_{\mu} \otimes \Hc_{\nu} 
\rightarrow \Hc_{\mu + \nu + 2k}$$
defined by
\begin{equation}
\label{Pk}
Q_{k}^{\mu, \nu}f(\xi)
\coloneqq C_{\mu,\nu,k} \sum_{j=0}^k (-1)^j \binom{k}{j} \frac{1}{(\mu)_j (\nu)_{k-j}} \partial^j_z \partial^{k-j}_w f|_{z=w=\xi}.
\end{equation}
Here $f\in \Hc_{\mu} \otimes \Hc_{\nu} $
is realized as holomorphic function $f(z, w)$
in $(z, w)\in D^2$, the constant is determined by
$$C_{\mu, \nu, k}^{-2} = 
\frac{k! (\mu + \nu + k + 1)_k}{(\mu)_k (\nu)_k},$$
so that $Q_k^{\mu,\nu}$ is a partial isometry \cite{vh-2}. 
We have also  \cite{hjal} 
$$ \Hc_{\nu}^{\otimes n} \cong \bigoplus_{k=0}^{\infty} 
\binom{n + k -2}{k} \Hc_{n \nu + 2k},$$
where we call the projection onto the $\binom{n + k -2}{k} \Hc_{n \nu + 2k}$-component $Q_k$.

\begin{lemma} 
Let  $f \in \Hc_{\nu}$.
The projection
$Q_1(f^{\otimes n} )$
of
 $f^{\otimes n} 
 \in \Hc_{\nu}^{\otimes n} \cong \bigotimes_{k=0}^{\infty} \binom{n+k-2}{n-2} \Hc_{n \nu + 2k}$ 
 onto the next leading
 component $(n-1) \Hc_{n \nu + 2}$ vanishes.
\end{lemma}

\begin{proof}
We do this by induction on $n\ge 2$. For $n=2$ we have
$$Q_{1}^{\nu,\nu}(f \otimes f) = C_{\nu,\nu,k} \left( \frac{1}{\nu} f'(\xi) f(\xi) - \frac{1}{\nu} f(\xi) f'(\xi) \right) = 0.$$
Now assume it is true for $\Hc_{\nu}^{\otimes n}$. We consider $\Hc_{\nu}^{\otimes (n+1)} = \Hc_{\nu}^{\otimes n} \otimes \Hc_{\nu}$, 
 the second
tensor factor is
$$ \Hc_{\nu}^{\otimes n}= \bigoplus_{k=0}^{\infty} \binom{n+k-2}{n-2} \Hc_{n \nu + 2k}$$
and $Q_{1}f^{\otimes n}=0$.
We look for the projection $F:=Q_{1}f^{\otimes (n+1)}
$ of $f^{\otimes (n+1)}
$ onto the
$n\Hc_{(n+1) \nu + 2}$-isotypic component.
It is obtained by
$$
F=Q_{1}^{n \nu, \nu}(Q_{0}f^{\otimes n}
\otimes f)+ Q_{0}^{n \nu + 2, \nu}(Q_{1}f^{\otimes n}
\otimes f)
=Q_{1}(Q_{0}f^{\otimes n}
\otimes f)
$$
which furthermore is 
\begin{equation*} F=
 C_{\nu, n \nu,1} \left( \frac{1}{\nu} f^{n}(\xi)f'(\xi) - \frac{1}{n \nu} n f'(\xi) f^{n}(\xi) \right) = 0,
\end{equation*}
completing the proof.
\end{proof}

\subsection{An improved
Wehrl inequality for the
unit disc}

Now we look at the projection onto the second factor $\Hc_{n \nu + 4}$ and obtain a stricter inequality.

\begin{theorem}
\label{improved-ineq}
We have the following improved
Wehrl $L^2-L^{2n}$ inequality 
\begin{equation}
\label{ineqest}
\nm{f^n}^2_{n \nu,2} 
+
\frac
{2 \nu^2(\nu+1)^2}
{(2 \nu + 3)(2 \nu + 4)} 
\nm
{
\left( \frac{f'' f}{(\nu)_2} 
- \frac{(f')^2}{\nu^2} \right) f^{n-2}
}_{n \nu + 4,2}^2
\leq \nm{ f}_{\nu,2}^{2n}.
\end{equation}
\end{theorem}

\begin{proof} We 
study the contribution to the component 
with highest weight
$\Hc_{n \nu + 4}$
in the tensor product $
f^{\otimes n} $.
There is
one contribution
obtained from $Q_2^{\nu,\nu}
(f \otimes f) 
\in \Hc_{2 \nu + 4} \subseteq \Hc_{\nu} \otimes \Hc_{\nu}$ 
and $f^{\otimes (n-2)} $,
$$ f^{\otimes n} \mapsto \left( z \mapsto f^{n-2}(z)
Q_{2\nu +4}^{\nu, \nu}
(f \otimes f)(z) \right).$$
Here
$Q_{2\nu +4}^{\nu, \nu}$ is the 
projection
\begin{equation*}\begin{split}
Q_{2\nu +4}^{\nu, \nu}(f \otimes f) &= \frac{\nu(\nu+1)}{\sqrt{2 (2 \nu + 3)(2 \nu + 4)}} 2 \left( \frac{1}{(\nu)_2} f'' f - \frac{1}{\nu^2} (f')^2 \right)
\\ & = \frac{\sqrt{2} \nu(\nu+1)}{\sqrt{(2 \nu + 3)(2 \nu + 4)}} \left( \frac{1}{(\nu)_2} f'' f - \frac{1}{\nu^2} (f')^2 \right).
\end{split}\end{equation*}
Thus we obtain 
\begin{equation*}\frac{\sqrt{2} \nu(\nu+1)}{\sqrt{(2 \nu + 3)(2 \nu + 4)}} \left( \frac{1}{(\nu)_2} f'' f - \frac{1}{\nu^2} (f')^2 \right)f^{n-2} \in \Hc_{n \nu + 4},
\end{equation*}
and 
\begin{equation*}
  \nm{f}_{\nu,2}^{2n} \geq \nm{f^n}^2_{n \nu,2} + \nm{\frac{\sqrt{2} \nu(\nu+1)}{\sqrt{(2 \nu + 3)(2 \nu + 4)}} \left( \frac{1}{(\nu)_2} f'' f - \frac{1}{\nu^2} (f')^2 \right) f^{n-2}}_{n \nu + 4,2}^2,
\end{equation*}
completing the proof.
\end{proof}

We note that 
the second summand in (\ref{ineqest}) 
is vanishing
exactly when $
\frac{1}{(\nu)_2} f'' f - \frac{1}{\nu^2} (f')^2 = 0$,
i.e.
$$
f'' f - \frac{\nu +1}{\nu} (f')^2 = 0.$$
The solutions are 
exactly the reproducing kernels $K_w$. Indeed
we can assume by $SU(1,1)$-invariance that $f(0) \neq 0$ and 
further  $f(0) = 1$. Suppose $f'(0)=c$.
Then we can recursively determine all the
derivatives $f^{(n)}(0)$ and find
$f^{(n)}(0) = \frac{(\nu)_n}{\nu^n} c^n
$. Thus $f(z) = 1  +\sum_{n=1}^\infty
\frac{(\nu)_n}{\nu^n} \frac{c^n }{n!} z^n$.
This is in the Bergman space if and only if $\frac{|c|}{\nu}<1$, in which
case $f(z)= K(z, w)$ with $w=\frac{c}{\nu}$.
We have thus found
a stronger inequality
and identified
the minimizer for the
extra summand.

\appendix
\section{Wehrl inequality
for matrix coefficients
of representations of compact Lie groups}
\label{cptwehrl}

We have used the 
Wehrl-inequality
for matrix coefficients
of representations
of compact Lie groups
in our proof of
the inequality
for non-compact
 hermitian symmetric spaces $D=G/K$. This inequality was
 stated in \cite{sugi}
 for tensor powers
 $V^{\otimes n}$
 for general $n\ge 2$, referring further back to \cite{DF} for $n=2$. However, the proof in \cite{DF} 
 is incomplete. 
The precise gap is that
the maximizer is proved to be an eigenvector of {\it one
 element} in the Cartan subalgebra but is claimed to be an eigenvector for {\it the whole Cartan subalgebra}. As we show below we do not actually need this fact, 
and 
 Cauchy-Schwarz inequality
 is need as  a critical step.

We  recall the Casimir operator. Let $\kf$ be the   Lie algebra
of a compact
semisimple Lie group $K$,
$\kf^{\C}$  the complexification
and $\kappa$ 
the Killing form. Let
$\{T_i\}$ be an orthonormal basis of $i\kf$ .
The Casimir element is 
$$C = \sum_i T_i^2 $$
and it acts on a representation
 $(V_{\tau}, \tau, K)$ of $K$
 as
$$ C = \sum_i \tau(T_i)^2.$$
 If $\tau$ is 
irreducible 
with highest weight $\Lambda$
then
$C$ acts on $V_{\Lambda}$ as the constant
$$ \langle \Lambda+\rho, \Lambda+\rho \rangle - \langle \rho, \rho \rangle,$$
where $\rho = \frac{1}{2} \sum_{\alpha \in \Delta^+} \alpha$ is the half-sum of the positive roots \cite[Exercise 23.4]{hum}.
(We have chosen $T_i$ to be
a basis of $i\kf$ so that
each $\tau(T_i)$ is
self-adjoint and $C$  non-negative.)

Recall further that if $V_{\Lambda_1}$ and $V_{\Lambda_2}$ are two finite-dimensional
irreducible representations of a semisimple Lie algebra with highest weights $\Lambda_1$ and
$\Lambda_2$ 
then
\begin{equation}
\label{tensor-2}
V_{\Lambda_1} \otimes V_{\Lambda_2} \cong \bigoplus_{\Lambda} V_{\Lambda},
\end{equation}
where $\Lambda \leq \Lambda_1 + \Lambda_2$ and $V_{\Lambda_1 + \Lambda_2}$ appears exactly once in the decomposition. This is clear from the weight space decomposition \cite[Chapter 21]{hum}.

\begin{proposition}
\label{wehrlcptprop}
\begin{enumerate}
\item Let $K$ be a compact Lie group 
and 
$(\tau_{\Lambda}, V_{\Lambda})$ 
an irreducible representation 
of highest
weight $\Lambda$.
Consider the irreducible
decomposition
$$V_{\Lambda}^{\otimes n} \cong 
\bigoplus_{\Lambda'} m_{\Lambda'}
V_{\Lambda'}$$
and 
$$v^{\otimes n} 
= \sum_{\Lambda'} v_{\Lambda'},
\,\, v_{\Lambda'}\in 
m_{\Lambda'}
V_{\Lambda'}$$
for 
$v\in V_{\Lambda}$ and $n\ge 2$.
We have $v^{\otimes n} \in V_{n\Lambda}$ if and only if $v = \tau(k) v_{\Lambda}$ for
some $k\in K$ and a highest weight vector $v_{\Lambda}$.
\item For any unit vector $v \in V_{\Lambda}$ 
the following Wehrl inequality holds,
\begin{equation}
\label{wehrlcpt}
 \int_{K} | \langle \tau(k) v, v_{\Lambda} \rangle|^{2n} d k \leq \frac{1}{\dim(V_{n \Lambda})},
\end{equation}
where $v_{\Lambda}$ is a unit highest weight vector, and the equality holds if and only if $v = \tau(k_0) v_{\Lambda}$ for some $k_0 \in K$ and some highest weight vector $v_{\Lambda}$.
\end{enumerate}
\end{proposition}

\begin{proof} 
We prove the Proposition only for compact semisimple Lie groups and it implies
the general result. We prove first that the second part is a consequence of
the first one.
Indeed, let
$$ P_{n \Lambda}: V_{\Lambda}^{\otimes n} \rightarrow V_{n \Lambda}$$
be the projection. Then  $w_{n \Lambda} = P_{n \Lambda}(v_{\Lambda}^{\otimes n}) \in V_{n \Lambda}$ is a highest weight vector of unit length and we have that $P_{n \Lambda}^*P_{n \Lambda}(v_{\Lambda}^{\otimes n}) = v_{\Lambda}^{\otimes n}$, implying $v_{\Lambda}^{\otimes n} \in V_{n \Lambda} \subseteq V_{\Lambda}^{\otimes n}$, so
\begin{equation*}\begin{split}
& \int_{K} | \langle \tau_{\Lambda}(k) v, v_{\Lambda} \rangle_{V_{\Lambda}} |^{2n} d k = \int_K |\langle (\tau_{\Lambda}(k) v)^{\otimes n}, v_{\Lambda}^{\otimes n} \rangle_{V_{\Lambda}^{\otimes n}}|^2 dk
\\ & = \int_K |\langle \tau_{n \Lambda}(k) P_{n \Lambda} (v^{\otimes n}), w_{\Lambda} \rangle_{V_{n \Lambda}}|^2 dk = \frac{1}{\dim(V_{n \Lambda})} \nm{P_{n \Lambda}(v^{\otimes n})}_{V_{n \Lambda}}^2.
\end{split}\end{equation*}
This immediately implies the inequality 
in (\ref{wehrlcpt}), and the equality holds
if and only if $v = \tau(k) v_{\Lambda}$ for a highest weight vector $v_{\Lambda}$ as $\nm{P_{n \Lambda}(v^{\otimes n})}_{V_{n \Lambda}} = \nm{P_{n \Lambda}(v_{\Lambda}^{\otimes n})}_{n \Lambda} = 1$ implies for any unit vector $v$ that $v \in V_{n \Lambda} \subseteq V_{\Lambda}^{\otimes n}$. It follows also that
\begin{equation}
 \int_{K} | \langle \tau(k) v, v_{\Lambda} \rangle|^{2n} d k \leq \int_K | \langle \tau(k) v_{\Lambda}, v_{\Lambda} \rangle|^{2n} d k,
\end{equation}
with equality if and only if $v = \tau(k) v_{\Lambda}$ for a highest weight vector $v_{\Lambda}$.

We now prove the first part. 
It follows
from
(\ref{tensor-2})
 that
 $V_{n \Lambda} \subseteq V_{\Lambda}^{\otimes n}$ with multiplicity one, and that the other representations appearing in the decomposition of $V_{\Lambda}^{\otimes n}$ have lower highest weights. Thus the sufficiency
 of the claim  is clear,
 and we prove the 
 necessity. 
 We prove it first for $n = 2$,
 so we assume that $v$ is a unit vector  
and  $v \otimes v \in V_{2 \Lambda} $.

Let $C = \sum_i T_i^2 $
be the Casimir element as above.
Fix a choice of a Cartan subalgebra $\h \subseteq \kf$.
Consider the
 decomposition
$$V_{\Lambda} \otimes V_{\Lambda} \cong \bigoplus m_{\Lambda'} V_{\Lambda'},$$
 the leading 
multiplicity being 
$m_{2 \Lambda} = 1$ and $\Lambda' \leq 2 \Lambda$. Hence $v \otimes v \in V_{2 \Lambda} 
\subseteq V_{\Lambda}^{\otimes 2}$ if and only if
\begin{equation}
\label{condV2l}
(\tau_{\Lambda} \otimes \tau_{\Lambda}) (C) ( v \otimes v ) = (\langle 2 \Lambda+\rho, 2 \Lambda+\rho \rangle - \langle \rho, \rho \rangle) v \otimes v.
\end{equation}
On the other hand
\begin{equation*}\begin{split}
 (\tau_{\Lambda} \otimes \tau_{\Lambda}) (C) &= \sum_i (\tau \otimes \tau)(T_i)^2 = \sum_i \left( \tau(T_i) \otimes I + I \otimes \tau(T_i) \right)^2
\\ & = \sum_i \tau(T_i)^2 \otimes I + I \otimes \tau(T_i)^2 + 2 \tau(T_i) \otimes \tau(T_i)
\\ & = \tau(C) \otimes I + I \otimes \tau(C) + 2 \sum_i \tau(T_i) \otimes \tau(T_i).
\end{split}\end{equation*}
Thus  for any $v$,
\begin{equation}
\label{tau-2-C-v} (\tau_{\Lambda} \otimes \tau_{\Lambda}) (C) ( v \otimes v ) = 2(\langle \Lambda+\rho, \Lambda+\rho \rangle - \langle \rho, \rho \rangle)(v \otimes v) + 2 \sum_i \tau(T_i) v \otimes \tau(T_i) v.
\end{equation}
Comparing 
(\ref{condV2l}) 
with (\ref{tau-2-C-v})
we find that $v \otimes v \in V_{2 \Lambda}$
if and only if
\begin{equation}
\label{eig-con} 
 \sum_i \tau(T_i) v \otimes \tau(T_i) v = \langle \Lambda, \Lambda \rangle v \otimes v.
 \end{equation}
By taking the first tensor factor we
see that
 (\ref{eig-con})
implies 
\begin{equation}
\label{eig-con-1} 
\sum_i \langle \tau(T_i) v, v \rangle \tau(T_i) v = \langle \Lambda, \Lambda \rangle v.
\end{equation}
Note that $\tau(T_i)$ is self-adjoint and thus this gives an element
$$ \sum_i \langle \tau(T_i) v, v \rangle T_i \in i \kf.$$
By \cite[Theorem 4.34]{knapp} there is a $k \in K$ such that
$$\Ad(k) \left( \sum_i \langle \tau(T_i) v, v \rangle T_i \right) \in i \h \subseteq \h^{\C},$$
and thus
\begin{equation}
\label{eig-con-2} 
 \Ad(k) \left( \sum_i \langle \tau(T_i) v, v \rangle T_i \right) = \sum_j a_j H_j
 \end{equation}
where $H_j$ is an orthonormal 
basis 
of $i\h$
w.r.t. the Killing form $\kappa$.
Note that
\begin{equation}
\begin{split}
\label{a-norm}
 \sum_j a_j^2 &= \kappa ( \sum_j a_j H_j, \sum_j a_j H_j )
\\ & = \kappa(\Ad(k) \left( \sum_i \langle \tau(T_i) v, v \rangle T_i \right), \Ad(k) \left( \sum_i \langle \tau(T_i) v, v \rangle T_i \right) )
\\ & = \kappa(\sum_i \langle \tau(T_i) v, v \rangle T_i, \sum_i \langle \tau(T_i) v, v \rangle T_i) = \sum_i \langle \tau(T_i) v, v \rangle^2
\\ & =  \langle \sum_i (\tau(T_i) \otimes \tau(T_i)) (v \otimes v), v \otimes v \rangle = \langle \Lambda, \Lambda \rangle.
\end{split}
\end{equation}

Applying $\tau(k)$
to (\ref{eig-con-1})
and using 
 (\ref{eig-con-2})
we find that $w:=\tau(k)v$ is an eigenvector of $\tau(\sum_j a_j H_j)$,
$$
\tau(\sum_j a_j H_j)
w=
\langle \Lambda, \Lambda \rangle 
w.
$$
Decompose
$$w=\tau(k)v = \sum w_{\Lambda'},
\quad 0 \neq w_{\Lambda'} \in W_{\Lambda'}
$$
as sum of weight vectors under the decomposition  $V_{\Lambda} = \bigoplus W_{\Lambda'}$ into weight subspaces of $\h^{\C}$. Now we get
\begin{equation*}
\langle \Lambda, \Lambda \rangle 
w =\tau(\sum_j a_j H_j)w = \sum_{\Lambda'} \left( \sum_j a_j \Lambda'(H_j) \right) w_{\Lambda'},
\end{equation*}
so that  $\sum_j a_j \Lambda'(H_j) = \langle \Lambda, \Lambda \rangle$.
The
Cauchy-Schwarz inequality
and (\ref{a-norm})
imply 
$$
\langle \Lambda, \Lambda \rangle=
\sum_j a_j \Lambda'(H_j) \leq \left( \sum_j a_j^2 \right)^{\frac{1}{2}} \left( \sum_j \Lambda'(H_j)^2 \right)^{\frac{1}{2}} = \langle \Lambda, \Lambda \rangle^{\frac{1}{2}} \langle \Lambda', \Lambda' \rangle^{\frac{1}{2}}.$$
But $\langle \Lambda', \Lambda' \rangle \le \langle \Lambda, \Lambda \rangle$
with equality only for $\Lambda' = \sigma \Lambda$ for some element $\sigma$ in the Weyl group $W$ by \cite[Theorem 5.5]{knapp}. So we find
all $w_{\Lambda'} = 0$ unless $\Lambda' = \sigma \Lambda$ for some $\sigma \in W$. Hence (by taking into account the
Weyl group element $\sigma$)
there is  $k \in K$ such that $\tau(k)v=w=w_\Lambda$ is a
highest weight vector.

Now we prove our claim  for general $n > 2$,
and we assume $v$ is a unit vector, $
v^{\otimes n} \in V_{n\Lambda}$.
Observe that $
v^{\otimes n}  
= v^{\otimes 2} \otimes v^{\otimes (n-2)}
\in V^{\otimes 2}
\otimes 
 V^{\otimes (n-2)}
$ and the first factor has
a decomposition 
\begin{equation}
\label{v-sq-dec}
v^{\otimes 2} = \sum_{\Lambda'\le 2\Lambda} v_{\Lambda'}
\in \bigoplus_{\Lambda'
\le 2\Lambda} m_{\Lambda'} V_{\Lambda'}.
\end{equation}
 For any $\Lambda'\leq 2\Lambda$, $\Lambda'\ne 2\Lambda$
we have
$$
V_{\Lambda'} 
\otimes V_{\Lambda}^{\otimes (n-2)}
=\sum_{\Lambda''} m(\Lambda'') V_{\Lambda''}
$$
with each $\Lambda'' < 
n\Lambda 
$.
Thus $V_{n\Lambda}\perp 
V_{\Lambda'} 
\otimes V_{\Lambda}^{\otimes (n-2)}
$ in $V^{\otimes n}$
and in particular $P_{n\Lambda}(v_{\Lambda'} \otimes v^{\otimes (n-2)}) =0.$
Therefore
$$
P_{n \Lambda}( v^{\otimes n}) = P_{n \Lambda} \left( ( \sum_{\Lambda'} v_{\Lambda'}) \otimes v^{\otimes (n-2)} \right) = P_{n \Lambda} (v_{2 \Lambda} \otimes v^{\otimes (n-2)}).$$
This implies that
$$
1=\Vert 
v^{\otimes n} 
\Vert =
\nm{P_{n \Lambda}(v^{\otimes n})} = \nm{P_{n \Lambda} (v_{2 \Lambda} \otimes v^{\otimes (n-2)})} \leq \nm{v_{2 \Lambda}} \cdot \nm{v}^{n-2} = \nm{v_{2 \Lambda}}\le 1.$$
Thus all other components
$v_{\Lambda'}$ in (\ref{v-sq-dec})
vanish and $
v^{\otimes 2}= 
v_{2 \Lambda} 
\in V_{2 \Lambda}$. 
This  reduces to the case
$n=2$ and completes the proof.
\end{proof}

\section{Bounded Point Evaluations for  Bergman Spaces of Vector-Valued Holomorphic Functions}
\label{BPV}
We prove what bounded point evaluations 
for our Bergman space
of vector valued
holomorphic functions on $D$
are given by points in $D$. 
This might be known fact for a larger
class of Bergman spaces but
we can not find some exact reference
and we present here an elementary proof.

\begin{lemma}
\label{evbdd}
Let $u\in \mathfrak p^+$, $0 \neq v_0 \in V_{\Lambda}$ be a fixed
vector, and consider
the evaluation of  polynomials $p \in \mathrm{Pol}(\mathbb C^N) \otimes V_{\Lambda}
\subset \Hc_{\Lambda}$,
$$ p \mapsto \langle p(u), v_0 \rangle_{\tau}.$$
It is  bounded on 
the Hilbert space $\mathcal H_{\Lambda}$  if and only if $u \in D$.
\end{lemma}

\begin{proof}
Obviously the evaluation map is bounded if $u \in D$ by the reproducing kernel property, as 
$$ | \langle p(u), v_0 \rangle_{\tau}| = |\langle p, K_u v_0 \rangle_{\Hc_{\Lambda}}| \leq \nm{K_u v_0}_{\Hc_{\Lambda}} \nm{p}_{\Hc_{\Lambda}}.$$
Now we prove the converse. 
Recall \cite{loos} that $\sum_{j=1}^r \R (e_j + e_{-j})
\subseteq \p$ is a  maximal Abelian
subalgebra of $\p$ and 
$$
a(t):=\exp(\sum_{j=1}^r t_j (e_j + e_{-j}))
: \cdot 0\mapsto x(t)=\sum_{j=1}^r \tanh t_j\,  e_j
$$
in $D=G/K$.
The space $\mathbb C^N$
is a disjoint union of $D=\{u \ | \ |u|<1\}$,
the boundary $\partial D=\{u \ | \ |u|=1\}$ and 
the complement $\overline D^{c}=\{u \ | \ |u|>1\}$.
We assume first that $u\in \partial D$ and 
prove that the evaluation
$$p \mapsto \langle p(u), v_0 \rangle_{\tau} $$
is unbounded. By the $K$-equivariance we can
assume that 
$u=u_1 e_1 + \dots + u_r e_r$,
with $u_1=1$ and $0\le u_j\le 1$. 
Write $x = x(t)=a(t) \cdot 0 =\sum_{j=1}^r x_j e_, \, x_j=x_j(t)=\tanh t_j, t_j\ge 0,$
as above.
Now if $\{v^s \}_s$ is an orthonormal basis of weight vectors for $V_{\Lambda}$ with weights $\Lambda^s$ then \cite{kora}
$$\tau(J_{a(t)}(0)) v^s = \prod_j (1-x_j^2)^{\frac{1}{2} \Lambda^s(h_j)} v^s.$$
As
$$B(g \cdot 0,\li{g \cdot 0}) = J_g(0) J_g(0)^*,$$
we see that
the reproducing kernel acts on $v^s$ as
$$  K(x,x) v^s = d_{\Lambda} \prod_j (1-x_j^2)^{ \Lambda^s(h_j)} v^s.$$
By using $K$-equivariance
we have 
the same is true for $z_1 e_1 + \dots + z_r e_r \in D$,  $z_1, \dots, z_r \in \C$, namely
$$ K(z,z) v^s = d_{\Lambda} \prod_j (1-|z_j|^2)^{ \Lambda^s(h_j)} v^s.$$
As $K(z,w)$ is holomorphic in the first coordinate and antiholomorphic in the second, we see that if $z = \sum_{j=1}^r z_j e_j, w = \sum_{j=1}^r w_j e_j \in D$ then
$$K(z,w) v^s = d_{\Lambda} \prod_j (1-z_j \li{w_j})^{ \Lambda^s(h_j)} v^s.$$
We have also for $0< \delta < 1$,
$$ K_{\delta w}(z)
= K(z,\delta w) 
=K(\delta z,w) 
=K_w(\delta z) 
$$
for any $z, w \in D$
and that the function $z \mapsto K_w(z)v$ has norm
$$ \nm{K_w v}_{\Hc_{\Lambda}}^2 = \langle K(w,w)v, v \rangle_{\tau}.$$
Thus  the function
$$f_{\delta u}: z \mapsto \langle K(\delta u,\delta u)v, v \rangle_{\tau}^{- \frac{1}{2}} K_{\delta u}(z) v$$
is of  unit norm and can be analytically extended to a bigger set
$$D_{\epsilon} = \{ z \in \mathbb C^N \ | \ d(z,D) < \epsilon \},$$
containing $\li{D}$ where the distance $d(z, D)$ is defined
using spectral norm. Now we choose a unit weight vector $v = v^s$ such that $| \langle v^s, v_0 \rangle | > 0$ and we get
\begin{equation*}\begin{split}
|\langle f_{\delta u}(u), v_0 \rangle_{\tau}| &= 
|\langle \langle K(\delta u, \delta u)v^s,v^s \rangle_{\tau}^{- \frac{1}{2}} K_{\delta u}(u) v^s, v_0 \rangle_{\tau}|
\\ & = d_{\Lambda}^{ \frac{1}{2}} ( \prod_j (1 - \delta^2 |u_j|^2)^{\Lambda^{s}(h_j)})^{- \frac{1}{2}} \prod_j (1- \delta |u_j|^2)^{\Lambda^s(h_j)} | \langle v^s, v_0 \rangle_{\tau}|
\\ & = d_{\Lambda}^{ \frac{1}{2}}  \prod_j 
\left( \frac{1- \delta |u_j|^2}{(1 - \delta^2 |u_j|^2)^{\frac{1}{2}}} \right)^{\Lambda^{s}(h_j)} | \langle v^s, v_0 \rangle_{\tau}|.
\end{split}\end{equation*}
Now by the Harish-Chandra condition in Theorem \ref{hccondition} we have that
$$ \Lambda^s(h_1) < 1-p \leq -1.$$
Also, $u_1 = 1$ so
$$ \lim_{\delta \rightarrow 1} \frac{1- \delta|u_1|^2}{(1-\delta^2 |u_1|^2 )^{\frac{1}{2}}} = \lim_{\delta \rightarrow 1} \frac{1- \delta}{(1-\delta^2 )^{\frac{1}{2}}} = 0.$$
It follows that
\begin{equation}
\label{limit}
\lim_{\delta \rightarrow 1} \langle f_{\delta u}(u), v_0 \rangle_{\tau} = \lim_{\delta \rightarrow 1} d_{\Lambda}^{ \frac{1}{2}}  \prod_j \left( \frac{1- \delta |u_j|^2}{(1 - \delta^2 |u_j|^2)^{\frac{1}{2}}} \right)^{\Lambda^{s}(h_j)} = \infty,
\end{equation}

Now the domain $D$ is convex, so $D$ and its closure $\bar D$
are polynomially convex. It follows by the Oka-Weil theorem \cite[Theorem VI.1.5]{rang} that for every $\epsilon > 0$ there are $V_{\Lambda}$-valued polynomials
$p_{k}$ on $\mathbb C^N$ 
such that 
$$ \sup_{z\in \overline D}\nm{f_{\delta u}(z) - p_{k}(z)}_{\tau} \to 0, k\to \infty.$$
By the dominated convergence
theorem we then also have
\begin{equation*}
\Vert
p_{k}
\Vert_{\Hc_{\Lambda}}^2 \to
\Vert
f_{\delta u}
\Vert_{\Hc_{\Lambda}}^2
=1, k\to \infty.
\end{equation*}
We can then use (\ref{limit}) to prove
that 
$\langle p(u), v_0 \rangle_{\tau}$
is unbounded for polynomials $p$.

Finally  we prove if $u \notin \li{D}$ then
evaluation is bounded. 
Clearly  $\frac{u}{|u|}
\in \partial D$ where ${|u|}$
is the spectral norm. 
Let $M > 0$ be arbitrarily
large. By the
previous result  for
 $\frac{u}{|u|}$ there 
 is a polynomial $p=p_M$  such 
 that $\nm{p} < 2$ and $\langle p(\frac{u}{|u|}), v_0 \rangle_{\tau} >M $. Denote $q(z) = p(\frac{z}{|u|}).$
Then for $z \in D$
\begin{equation*}
 \nm{ q(z) }_{\tau} = \nm{p(\frac{z}{|u|})}_{\tau} = \langle  p, K_{\frac{z}{|u|}}\frac{q(z)}{\nm{q(z)}_{\tau}} \rangle_{\Hc_{\Lambda}} \leq \nm{p}_{\Hc_{\Lambda}} \nm{K_{\frac{z}{|u|}}}_{\Hc_{\Lambda}}.
\end{equation*}
This must be bounded as $\nm{p} \leq 2$ and $z \mapsto \nm{K_{\frac{z}{|u|}}}$ is bounded on $D$ as $|u| > 1$. Thus $\nm{q}$ is also bounded irrespective of $M$. However,
$$\langle q(u), v_0 \rangle > M,$$
so evaluation in $u$ is unbounded. This completes
the proof.
\end{proof}


\begin{thebibliography}{99}

\bibitem{AAR} G.E. Andrews, R. Askey and  R. Roy, \textit{Special functions}, Encyclopedia Math. Appl., vol. 71,
Cambridge University Press, 1999.

\bibitem{DF}
   R.~Delbourgo  and J.~Fox,
   \textit{
     Maximum weight vectors possess minimal
uncertainty}, 
J. Phys. A: Math. Gen.  \textbf{10} (1977), L233-L235. 


\bibitem{EHW} T. Enright, R. Howe and N. Wallach, \textit{A classification of unitary highest weight modules}, in \textit{Representation theory of reductive groups}, 97-144, Progress in Mathematics, 40, Birkh\"auser Boston, (1983).

\bibitem{fako} J. Faraut and A. Kor\'anyi, \textit{Function spaces and reproducing kernels on bounded symmetric domains}, J. Funct Anal. \textbf{88} (1990), no. 1, 64-89. 

\bibitem{Fr} R.~L.~Frank,
  \textit{
    Sharp inequalities for coherent states and their optimizers},
  preprint, arXiv:2210.14798

\bibitem{FNT} 
R.~L.~Frank,
F.~Nicola and  P.~Tilli,
  \textit{
The generalized Wehrl entropy bound in quantitative form,}
arXiv:2307.14089.

\bibitem{vH-1} R.~van Haastrecht, \textit{Limit formulas for the trace of the functional calculus of quantum channels for $SU(2)$}, J. Lie Theory. \textbf{34} (2024), no.3, 653-676.

\bibitem{vh-2} R.~van Haastrecht, \textit{Functional calculus of quantum channels for the holomorphic discrete series of $SU(1,1)$,} arXiv:2408.13083.

\bibitem{hari}
Harish-Chandra, \textit{Representations of semisimple Lie Groups VI: integrable and square-integrable representations}, Amer. J. Math. \textbf{78} (1956), no. 3, 564-628.

\bibitem{helgGGA} S. Helgason, \textit{Groups and geometric analysis}, Mathematical Surveys and Monographs vol. 83, Amer. Math. Soc., 
1984.



\bibitem{hum} J.E. Humphreys, \textit{Introduction to Lie algebras and representation theory}, Grad. Texts in Math., vol. 9, Springer, 1972.

\bibitem{knapp} A.~W. Knapp, \textit{Lie groups beyond an introduction}, Progr. Math., vol. 140, Birkh\"a{}user, 2002.

\bibitem{kora} A. Kor\'anyi, \textit{A simplified approach to the holomorphic discrete series}, preprint, arXiv:2312.16350.

\bibitem{ku} 
A.~Kulikov,
  \textit{
    Functionals with extrema at reproducing kernels, }
Geom. Funct. Anal.  \textbf{32} (2022), no. 4, 938-949. 


\bibitem{Lieb-cmp} 
E.~Lieb,
  \textit{
  Proof of an entropy conjecture of Wehrl,}
Comm. Math. Phys.
{\bf
62
}
(1978), no.1, 35–41.


\bibitem{LS-acta} 
E.~Lieb 
and J.~P.~Solovej,
  \textit{
    Proof of an entropy conjecture for Bloch coherent spin states and its generalizations,}
Acta Math.  
{\bf
212
} (2014), no. 2, 379--398. 

\bibitem{LS-cmp}E.~Lieb and J.~P.~Solovej,
\textit{
  Proof of the Wehrl-type entropy conjecture for symmetric SU(N) coherent states,
}
  Comm. Math. Phys.
  {\bf
  348
  } (2014), no. 2, 567--578.

\bibitem{LS-aff}
  E.~Lieb and J.~P.~Solovej,
  \textit{Proof of a Wehrl-type entropy
    inequality for the affince $AX+B$ group},
 EMS Ser. Congr. Rep.
EMS Press, Berlin, (2021), 301–314.



\bibitem{loos} O. Loos, \textit{Bounded symmetric domains and Jordan pairs}, University of California, Irvine (1977).

\bibitem{PZ}
  L.~Peng and G.~Zhang, \textit{
    Tensor products of holomorphic 
  representations and bilinear differential operators}, J. Funct. Anal. 
\textbf{210} (2004), no.~1, 171--192.

\bibitem{rang} R.M. Range, \textit{Holomorphic functions and integral representations in several complex variables}, Grad. Texts in Math., vol. 108, Springer, 1986.

\bibitem{repka} J. Repka, \textit{Tensor products of holomorphic discrete series representations}, Canadian J. Math. \textbf{31} (1979), no. 4, 836-844.

\bibitem{repkasl2} J. Repka, \textit{Tensor products of unitary representations of $SL_2(\R)$}, Amer. J. Math. \textbf{100} (1978), no.4, 747-774.

\bibitem{hjal} H. Rosengren, \textit{Multivariate orthogonal polynomials and coupling coefficients for discrete series representations}, SIAM J. Math. Anal. \textbf{30} (1999), no. 2, 232-272.

\bibitem{sata} I. Satake, \textit{Algebraic structures of symmetric domains},
Kan\^o Memorial Lectures, vol. 4, 
Iwanami Shoten, Tokyo; Princeton University Press, Princeton, NJ, 1980.

\bibitem{sugi} A. Sugita, \textit{Proof of the generalized Lieb-Wehrl conjecture for integer indices larger than one}, J. Phys. A \textbf{35} (2002), no.42, L621-626.

\bibitem{schli} H. Schlichtkrull, \textit{One-dimensional $K$-types in finite-dimensional representations of semisimple Lie groups: a generalization of Helgason's theorem}, Math. Scand. \textbf{54} (1984), no.2, 279–294.

\bibitem{schmid} W. Schmid, \textit{On the characters of the discrete series: The Hermitian symmetric case}, Invent. Math. \textbf{30} (1975), no.1, 47–144.

\bibitem{Wa} 
N.~R.~Wallach, 
\textit{
The analytic continuation of the discrete series. I, II
,}
Trans. Amer. Math. Soc.
{\bf
251
}
(1979), 1-17, 19-37.


\bibitem{weh} 
A. Wehrl,
\textit{
On the relation between classical and quantum-mechanical entropy,}
Rept. Math. Phys.
{\bf
16
}
(1979), 
no. 3, 353-358.




\bibitem{Z-24} 
G. Zhang, 
\textit{
Wehrl-type inequalities for Bergman spaces on domains in $\C^d$ and completely positive maps,}
in {\it The Bergman kernel and related topics}, 343-355, Springer Proc. Math. Stat., 447, Springer, Singapore, 2024.

\end{thebibliography}
\end{document}